\documentclass[11pt,a4paper]{article}
\usepackage[english]{babel}
\usepackage{exscale,times}

\usepackage{amssymb}
\usepackage{amsmath}
\usepackage{psfrag}
\usepackage{ifthen}
\usepackage{graphicx}
\usepackage{color}

\usepackage{fullpage}

\newenvironment{proof}[1][]
{
   \noindent\textrm{\bf Proof%
   \ifthenelse{\equal{#1}{}}{:}{~#1:} }\rm
}
{
   \hfill$\square$
   \bigskip
}

\newtheorem{theorem}{Theorem}[section]
\newtheorem{proposition}[theorem]{Proposition}
\newtheorem{lemma}[theorem]{Lemma}
\newtheorem{corollary}[theorem]{Corollary}
\newtheorem{definition}[theorem]{Definition}
\newtheorem{remark}[theorem]{Remark}

\newcommand{\R}[0]{\mathbb{R}}

\renewcommand{\H}[0]{\mathcal{H}}

\newcommand{\ol}[1]{\overline{#1}}

\newcommand{\ra}[0]{\rightarrow}

\newcommand{\abs}[1]{\left\vert #1 \right\vert}
\newcommand{\norm}[1]{\left\| #1 \right\|}

\newcommand{\bea}{\begin{eqnarray*}}
\newcommand{\eea}{\end{eqnarray*}}
\newcommand{\bean}{\begin{eqnarray}}
\newcommand{\eean}{\end{eqnarray}}
\newcommand{\been}{\begin{equation}}
\newcommand{\een}{\end{equation}}
\newcommand{\bee}{\begin{equation*}}
\newcommand{\ee}{\end{equation*}}

\newcommand{\bdefi}{\begin{definition}}
\newcommand{\edefi}{\end{definition}}
\newcommand{\bthe}{\begin{theorem}}
\newcommand{\ethe}{\end{theorem}}
\newcommand{\bpro}{\begin{proof}}
\newcommand{\epro}{\end{proof}}
\newcommand{\blem}{\begin{lemma}}
\newcommand{\elem}{\end{lemma}}
\newcommand{\bcor}{\begin{corollary}}
\newcommand{\ecor}{\end{corollary}}
\newcommand{\brem}{\begin{remark}}
\newcommand{\erem}{\end{remark}}
\newcommand{\bpropo}{\begin{proposition}}
\newcommand{\epropo}{\end{proposition}}

\newcommand{\bsatz}{\begin{satz}}
\newcommand{\esatz}{\end{satz}}

\newcommand{\bbsp}{\begin{beispiel}\rm}
\newcommand{\ebsp}{\end{beispiel}}

\newcommand{\bbem}{\begin{bemerkung}}
\newcommand{\ebem}{\end{bemerkung}}

\newcommand{\bpm}{\begin{pmatrix}}
\newcommand{\epm}{\end{pmatrix}}

\renewcommand{\div}{\text{div}}

\renewcommand{\epsilon}{\varepsilon}

\newcommand{\supp}[1]{\text{supp}(#1)}

\renewcommand{\supp}{\operatorname*{supp}}
\newcommand{\skp}[1]{\left< #1 \right>}
\newcommand{\T}{\mathcal{T}}
\newcommand{\depth}{\text{depth}}

\newcommand{\lefttriplenorm}{\ensuremath{\left| \! \left| \! \left|}}
\newcommand{\righttriplenorm}{\ensuremath{\right| \! \right| \! \right|}}
\newcommand{\triplenorm}[1]{\lefttriplenorm #1 \righttriplenorm}

\newcommand{\level}{\text{level}}

\newcommand{\Pfar}{P_{\text{far}}}

\newcommand{\diam}{\operatorname*{diam}}
\newcommand{\dist}{\operatorname{dist}}

\newcommand{\Cdim}{C_{\dim}}

\usepackage{enumerate}
\usepackage{comment}

\title{$\H$-matrix approximability of the inverses of FEM matrices}
\author{Markus Faustmann \and Jens Markus Melenk \and Dirk Praetorius}

\begin{document}

\maketitle

\abstract{
We study the question of approximability for the inverse of the FEM stiffness matrix 
for (scalar) second order elliptic boundary value problems
by blockwise low rank matrices such as 
those given by the $\mathcal{H}$-matrix format introduced in~\cite{Hackbusch2}.
We show that 
exponential convergence in the local block rank $r$ can be achieved.  
We also show that exponentially accurate $LU$-decompositions in the $\H$-matrix format 
are possible for the stiffness matrices arising in the FEM. 
Unlike prior works, our analysis avoids any coupling of the block rank $r$ and
the mesh width $h$ and also covers mixed Dirichlet-Neumann-Robin boundary conditions.
}

\section{Introduction}
The format of ${\mathcal H}$-matrices was introduced in~\cite{Hackbusch2} as  
blockwise low-rank matrices that permit storage, application, and even a full (approximate) 
arithmetic with log-linear complexity, \cite{GrasedyckDissertation,GrasedyckHackbusch,Hackbusch}. 
This data-sparse format 
is well suited to represent at high accuracy matrices arising as discretizations of many integral
operators, for example, those appearing in boundary integral equation methods. Also the sparse matrices
that are obtained when discretizing differential operator by means of the finite element method (FEM)
are amenable to a treatment by $\H$-matrices; in fact, they feature a lossless representation. 
Since the $\H$-matrix format
comes with an arithmetic that provides algorithms to invert matrices as well as 
to compute $LU$-factorizations, approximations of the inverses of FEM matrices or their $LU$-factorizations
are available computationally. Immediately, the question of accuracy and/or complexity comes into sight. 
On the one hand, the complexity of the $\H$-matrix inversion can be log-linear if the $\H$-matrix structure
including the block ranks is fixed,
\cite{GrasedyckDissertation,GrasedyckHackbusch,Hackbusch}. 
Then, however, the accuracy of the resulting approximate inverse 
is not completely clear. On the other hand, the accuracy of the inverse can be controlled by means of an 
adaptive arithmetic (going back at least to \cite{GrasedyckDissertation}); the computational cost at which
this error control comes, is problem-dependent and not completely clear. Therefore, 
a fundamental question is how well the inverse can be approximated in a selected $\H$-matrix format, 
irrespective of algorithmic considerations. This question is answered in the present paper for 
FEM matrices arising from the discretization of second order elliptic boundary value problems. 

It was first observed numerically in~\cite{GrasedyckDissertation} that the
inverse of the finite element (FEM) stiffness matrix corresponding to the Dirichlet problem for 
elliptic operators with bounded coefficients can be approximated in the format of $\mathcal{H}$-matrices with an error 
that decays exponentially in the block rank employed . Using properties of the continuous Green's 
function for the Dirichlet problem,~\cite{BebendorfHackbusch} 
proves this exponential decay in the block rank, at least up to the discretization error. The work~\cite{Boerm} 
improves on the result~\cite{BebendorfHackbusch} in several ways, in particular, by proving a corresponding
approximation result in the framework of ${\mathcal H}^2$-matrices; we do not go into the details
of $\H^2$-matrices here and merely mention that $\H^2$-matrices 
are a refinement of the concept of $\H$-matrices with better complexity properties, 
\cite{Giebermann,hackbusch-khoromskij-sauter00,hackbusch-boerm02,BoermBuch}. 

Whereas the analysis of~\cite{BebendorfHackbusch,Boerm} is based on the solution operator on the 
continuous level (i.e., by studying the Green's function), the approach taken in the present article is to work
on the discrete level.  This seemingly technical difference
has several important ramifications: First, the exponential approximability
in the block rank shown here is not limited by the discretization error as in \cite{BebendorfHackbusch,Boerm}.
Second, in contrast 
to \cite{BebendorfHackbusch,Boerm}, where the block rank $r$ and the mesh width $h$ are coupled 
by $r \sim \abs{\log h}$, our estimates are explicit in both $r$ and $h$. Third, a unified treatment of 
a variety of boundary conditions is possible and indeed worked out by us. Fourth, our approach paves the way 
for a similar approximability result for discretizations of boundary integral operators, \cite{FaustmannMelenkPraetorius13}. 
Additionally, we mention
that we also allow here the case of higher order FEM discretizations. 

The last theoretical part of this paper (Section~\ref{sec:hierarchical-LU-decomposition})
shows that the ${\mathcal H}$-matrix format admits ${\mathcal H}$-$LU$-decompositions 
or ${\mathcal H}$-Cholesky factorizations with exponential accuracy in the block rank. 
This is achieved, following \cite{Bebendorf07,chandrasekaran-dewilde-gu-somasuderam10}, 
by exploiting that the off-diagonal blocks of certain Schur complements are low-rank. 
Such an approach is closely related to the concepts of 
hierarchically semiseparable matrices (see, for example, 
\cite{xia13,xia-chandrasekaran-gu-li09,li-gu-wu-xia12} and references therein)
and recursive skeletonization (see \cite{ho-greengard12,greengard-gueyffier-martinsson-rokhlin09}) and their arithmetic. 
In fact, several multilevel ``direct'' solvers for PDE discretizations have been proposed 
in the recent past, \cite{ho-ying13,gillman-martinsson13,schmitz-ying12,martinsson09}. 
These solvers take the form of (approximate) matrix factorizations. 
A key ingredient to their efficiency is that certain Schur complement blocks are compressible 
since they are low-rank. Thus, our analysis 
in Section~\ref{sec:hierarchical-LU-decomposition} could also be of value for the understanding of these 
algorithms. We close by stressing that our analysis in 
Section~\ref{sec:hierarchical-LU-decomposition} of ${\mathcal H}$-$LU$-decompositions makes
very few assumptions on the actual ordering of the unknowns and does not
explore beneficial features of special orderings. It is well-known 
in the context of classical direct solvers that the ordering of the unknowns has a tremendous
impact on the fill-in in factorizations. One of the most successful techniques for 
discretizations of PDEs are multilevel nested dissection strategies, which permit to 
identify large matrix blocks that will not be filled during the factorization. 
An in-depth complexity analysis for the ${\mathcal H}$-matrix arithmetic for such 
ordering strategies can be found in \cite{GrasedyckKriemannLeBorne}. The recent works 
\cite{ho-ying13,gillman-martinsson13} and, 
in a slightly different context, \cite{bennighof-lehoucq04}, owe at least parts of their 
efficiency to the use of nested dissection techniques. 

 
\section{Main results}
Let $\Omega \subset \mathbb{R}^{d}$, $d \in \{2,3\}$, be a bounded polygonal (for $d = 2$) or polyhedral (for $d=  3$)
Lipschitz domain with boundary $\Gamma:=\partial\Omega$. 
We consider differential operators of the form
\been\label{eq:ellipticoperatorLO}
Lu := - \div(\mathbf{C}\nabla u)+ \mathbf{b}\cdot \nabla u + \beta u, 
\een
where $\mathbf{b} \in L^\infty(\Omega ; \R^{d}),\beta \in L^{\infty}(\Omega)$,
and $\mathbf{C} \in L^\infty(\Omega ; \R^{d\times d})$ is pointwise symmetric with 
\been
c_1\norm{y}_2^2 \leq  \skp{\mathbf{C}(x)y,y}_2 \leq c_2 \norm{y}_2^2 \quad \forall y \in \R^d,
\een
with certain constants $c_1, c_2 >0$. 

For $f\in L^2(\Omega)$, we consider the mixed boundary value problem
\begin{subequations}\label{eq:modelstrong}
\begin{align}
Lu &= f \quad \text{in} \; \Omega, \\
u &= 0 \quad \text{on} \; \Gamma_D,  \\
\mathbf{C}\nabla u \cdot n &= 0 \quad \text{on} \; \Gamma_N, \\
\mathbf{C}\nabla u \cdot n + \alpha u &= 0 \quad \text{on} \; \Gamma_{\mathcal{R}},
\end{align}
\end{subequations}
where $n$ denotes the outer normal vector to the surface $\Gamma$, $\alpha \in L^{\infty}(\Gamma_{\mathcal{R}})$, $\alpha >0 $ and 
$\Gamma = \ol{\Gamma_D} \; \cup \; \ol{\Gamma_N} \; \cup \; \ol{\Gamma_{\mathcal{R}}}$, 
with the pairwise disjoint and relatively open
subsets $\Gamma_D,\Gamma_N,\Gamma_{\mathcal{R}}$. 
With the trace operator $ \gamma_0^{\text{int}}$ we define $H^1_0(\Omega,\Gamma_D):= \{u\in H^1(\Omega): \gamma_0^{\text{int}}u = 0 \;\text{on}\; \Gamma_D\}$.
The bilinear form $a: H^1_0(\Omega,\Gamma_D) \times H^1_0(\Omega,\Gamma_D) \ra \R$ corresponding to \eqref{eq:modelstrong}
 is given by 
\been\label{eq:variation}
a(u,v) := \skp{\mathbf{C}\nabla u, \nabla v}_{L^2(\Omega)} + \skp{\mathbf{b}\cdot \nabla u+\beta u, v}_{L^2(\Omega)} + \skp{\alpha u, v}_{L^2(\Gamma_{\mathcal{R}})}.
\een
We additionally assume that the coefficients  $\alpha,\mathbf{C},\mathbf{b},\beta $ are such that the
the coercivity 
\been\label{eq:coercivity}
\norm{u}_{H^1(\Omega)}^2\leq C a(u,u)
\een
of the bilinear form $a(\cdot,\cdot)$ holds. Then, the Lax-Milgram Lemma implies the unique solvability of the weak
formulation of our model problem. \\

For the discretization, we assume that $\Omega$ is triangulated by 
a {\em quasiuniform} mesh ${\mathcal T}_h=\{T_1,\dots,T_N\}$ of mesh width 
$h := \max_{T_j\in \mathcal{T}_h}{\rm diam}(T_j)$, and the Dirichlet $\Gamma_D$, Neumann $\Gamma_N$, 
and Robin $\Gamma_{\mathcal{R}}$-parts of the boundary are resolved by the mesh $\T_h$.
The elements $T_j \in \mathcal{T}_h$ are 
triangles ($d=2$) or tetrahedra ($d =3$), and we assume that $\mathcal{T}_h$ is regular in the sense of Ciarlet. 
The nodes 
are denoted by $x_i \in \mathcal{N}_h$, for $i=1,\ldots,N$. 
Moreover, the mesh $\mathcal{T}_h$ is assumed to be $\gamma$-shape regular in the sense 
of $ h \sim {\rm diam}(T_j) \le \gamma\,|T_j|^{1/d}$ for all $T_j\in\mathcal{T}_h$.
In the following, the notation $\lesssim$ abbreviates $\leq$ up to a 
constant $C>0$ which depends only on $\Omega$, the dimension $d$, and 
$\gamma$-shape regularity of $\mathcal{T}_h$. 
Moreover, we use $\simeq$ to abbreviate that both estimates
$\lesssim$ and $\gtrsim$ hold.\\

We consider the Galerkin discretization of the bilinear form 
$a(\cdot,\cdot)$ by continuous, piecewise polynomials of fixed degree $p \geq 1$ in 
$S^{p,1}_0({\mathcal T}_h,\Gamma_D) := S^{p,1}({\mathcal T}_h) \cap H^1_0(\Omega,\Gamma_D)$ with 
$S^{p,1}({\mathcal T}_h) = \{u \in C(\Omega)\,:\, u|_{T_j} \in \mathcal{P}_p, \, \forall \, T_j \in \mathcal{T}_h\}$.
We choose a basis of $S^{p,1}_0({\mathcal T}_h,\Gamma_D)$, which is denoted by
 ${\mathcal B}_h:= \{\psi_j\, :\, j = 1,\dots, N\}$. 
Given that our results are formulated for matrices, assumptions on the basis ${\mathcal B}_h$ 
need to be imposed. For the 
isomorphism $\mathcal{J}:\R^N\ra S^{p,1}_0({\mathcal T}_h,\Gamma_D)$, $\mathbf{x} \mapsto \sum_{j=1}^Nx_j\psi_j$, 
we require 
\been\label{eq:basis}
h^{d/2}\norm{\mathbf{x}}_2 \lesssim \norm{\mathcal{J}\mathbf{x}}_{L^2(\Omega)} \lesssim h^{d/2}\norm{\mathbf{x}}_2,
\quad \forall\, \mathbf{x} \in \R^d. 
\een

\begin{remark}
Standard bases for $p=1$ are the classical hat functions satisfying $\psi_j(x_i) = \delta_{ij}$ and for $p\geq 2$ 
we refer to, e.g., \cite{SchwabBuch,karniadakis-sherwin99,demkowicz-kurtz-pardo-paszynski-rachowicz-zdunek08}.
\end{remark}

The Galerkin discretization of \eqref{eq:variation} results in a positive definite matrix 
$\mathbf A\in\mathbb R^{N\times N}$ 
with
\begin{equation*}
 \mathbf A_{jk} = \skp{\mathbf{C} \nabla \psi_k,\nabla \psi_j}_{L^2(\Omega)} + \skp{\mathbf{b}\cdot \nabla \psi_k+\beta\psi_k,\psi_j }_{L^2(\Omega)}
+ \skp{\alpha \psi_k,\psi_j}_{L^2(\Gamma_{\mathcal{R}})}, \quad \psi_k,\psi_j \in {\mathcal B}_h. 
\end{equation*}

Our goal is to derive an $\mathcal{H}$-matrix approximation $\mathbf{B}_{\mathcal{H}}$ of the inverse matrix $\mathbf{B} =\mathbf{A}^{-1}$.
An ${\mathcal H}$-matrix $\mathbf{B}_{\mathcal{H}}$ is a blockwise low rank matrix based on the concept of ``admissibility'', 
which we now introduce: 

\begin{definition}[bounding boxes and $\eta$-admissibility]
\label{def:admissibility}
A \emph{cluster} $\tau$ is a subset of the index set $\mathcal{I} = \{1,\ldots,N\}$. For a cluster $\tau \subset \mathcal{I}$, 
we say that $B_{R_{\tau}} \subset \mathbb{R}^d$ is a \emph{bounding box} if: 
\begin{enumerate}[(i)]
 \item 
\label{item:def:admissibility-i}
$B_{R_{\tau}}$ is a hyper cube with side length $R_{\tau}$,
 \item 
\label{item:def:admissibility-ii}
$ \supp \psi_j \subset B_{R_{\tau}}$ for all $ j \in \tau$.
\end{enumerate}

For $\eta > 0$, a pair of clusters $(\tau,\sigma)$ with $\tau,\sigma \subset \mathcal{I}$ is $\eta$-\emph{admissible}, if 
there exist boxes $B_{R_{\tau}}$, $B_{R_{\sigma}}$ satisfying 
(\ref{item:def:admissibility-i})--(\ref{item:def:admissibility-ii})
such that
\begin{equation}\label{eq:admissibility}
\max\{{\rm diam}B_{R_{\tau}},{\rm diam}B_{R_{\sigma}}\} \leq  \eta \; {\rm dist}(B_{R_{\tau}},B_{R_{\sigma}}).
\end{equation}
\end{definition}

\begin{definition}[blockwise rank-$r$ matrices]
Let $P$ be a partition of ${\mathcal I} \times {\mathcal I}$ and $\eta >0$. A matrix ${\mathbf B}_{\mathcal{H}} \in \mathbb{R}^{N \times N}$ 
is said to be a \emph{blockwise rank-$r$ matrix}, if for every $\eta$-admissible cluster pair $(\tau,\sigma) \in P$, 
the block ${\mathbf B}_{\mathcal{H}}|_{\tau \times \sigma}$ is a rank-$r$ matrix, i.e., it has the form 
${\mathbf B}_{\mathcal{H}}|_{\tau \times \sigma} = {\mathbf X}_{\tau \sigma} {\mathbf Y}^T_{\tau \sigma}$ with 
$\mathbf{X}_{\tau\sigma} \in \mathbb{R}^{\abs{\tau}\times r}$ 
and $\mathbf{Y}_{\tau\sigma} \in \mathbb{R}^{\abs{\sigma}\times r}$.
Here and below, $\abs{\sigma}$ denotes the cardinality of a finite set $\sigma$.
\end{definition}

The following theorems are the main results of this paper. Theorem~\ref{th:blockapprox} shows that admissible blocks
can be approximated by rank-$r$ matrices:

\begin{theorem}\label{th:blockapprox}
Fix $\eta > 0$, $q\in (0,1)$. Let the cluster pair $(\tau,\sigma)$ be $\eta$-admissible. Then, for $k \in \mathbb{N}$ there are matrices
$\mathbf{X}_{\tau\sigma} \in \mathbb{R}^{\abs{\tau}\times r}$, $\mathbf{Y}_{\tau\sigma} \in \mathbb{R}^{\abs{\sigma}\times r}$ of rank $r \leq C_{\rm dim} (2+\eta)^dq^{-d}k^{d+1}$
such that
\begin{equation}
\norm{\mathbf{A}^{-1}|_{\tau \times \sigma} - \mathbf{X}_{\tau\sigma}\mathbf{Y}_{\tau\sigma}^T}_2 \leq 
C_{\rm apx} N q^k.
\end{equation}
The constants $C_{\rm apx},C_{\rm dim}>0$ depend only on the boundary value problem \eqref{eq:modelstrong}, 
$\Omega$, $d$, $p$, and the $\gamma$-shape regularity of $\mathcal{T}_h$.
\end{theorem}
The approximations for the individual blocks can be combined to gauge the approximability of ${\mathbf A}^{-1}$
by blockwise rank-$r$ matrices. Particularly satisfactory estimates are obtained if the blockwise rank-$r$ matrices
have additional structure. To that end, we introduce the following definitions. 

\begin{definition}[cluster tree]
A \emph{cluster tree} with \emph{leaf size} $n_{\rm leaf} \in \mathbb{N}$ is a binary tree $\mathbb{T}_{\mathcal{I}}$ with root $\mathcal{I}$  
such that for each cluster $\tau \in \mathbb{T}_{\mathcal{I}}$ the following dichotomy holds: either $\tau$ is a leaf of the tree and 
$\abs{\tau} \leq n_{\rm leaf}$, or there exist so called sons $\tau'$, $\tau'' \in \mathbb{T}_{\mathcal{I}}$, which are disjoint subsets of $\tau$ with 
$\tau = \tau' \cup \tau''$. The \emph{level function} ${\rm level}: \mathbb{T}_{\mathcal{I}} \rightarrow \mathbb{N}_0$ is inductively defined by 
${\rm level}(\mathcal{I}) = 0$ and ${\rm level}(\tau') := {\rm level}(\tau) + 1$ for $\tau'$ a son of $\tau$. The \emph{depth} of a cluster tree
is ${\rm depth}(\mathbb{T}_{\mathcal{I}}) := \max_{\tau \in \mathbb{T}_{\mathcal{I}}}{\rm level}(\tau)$.  
\end{definition}

\begin{definition}[far field, near field, and sparsity constant]
 A partition $P$ of $\mathcal{I} \times \mathcal{I}$ is said to be based on the cluster tree $\mathbb{T}_{\mathcal{I}}$, 
if $P \subset \mathbb{T}_{\mathcal{I}}\times\mathbb{T}_{\mathcal{I}}$. For such a partition $P$ and fixed $\eta > 0$, we define the \emph{far field} and the \emph{near field} 
as 
\begin{equation*}
P_{\rm far} := \{(\tau,\sigma) \in P \; : \; (\tau,\sigma) \; \text{is $\eta$-admissible}\}, \quad P_{\rm near} := P\backslash P_{\rm far}.
\end{equation*}
The \emph{sparsity constant} $C_{\rm sp}$, introduced in \cite{GrasedyckDissertation}, of such a partition is defined by
\begin{equation*}
C_{\rm sp} := \max\left\{\max_{\tau \in \mathbb{T}_{\mathcal{I}}}\abs{\{\sigma \in \mathbb{T}_{\mathcal{I}} \, : \, \tau \times \sigma \in P_{\rm far}\}},\max_{\sigma \in \mathbb{T}_{\mathcal{I}}}\abs{\{\tau \in \mathbb{T}_{\mathcal{I}} \, : \, \tau \times \sigma \in P_{\rm far}\}}\right\}.
\end{equation*}
\end{definition}

The following Theorem~\ref{th:Happrox} shows that the matrix
$\mathbf{A}^{-1}$ can be approximated by blockwise rank-$r$ matrices at an exponential rate in the 
block rank $r$: 

\begin{theorem}\label{th:Happrox}
Fix $\eta > 0$. Let a partition $P$ of ${\mathcal{I}} \times {\mathcal{I}}$ be based on a cluster tree 
$\mathbb{T}_{\mathcal{I}}$. Then, there is a 
blockwise rank-$r$ matrix $\mathbf{B}_{\mathcal{H}}$ such that 
\begin{equation}
\norm{\mathbf{A}^{-1} - \mathbf{B}_{\mathcal{H}}}_2 \leq C_{\rm apx} C_{\rm sp} N {\rm depth}(\mathbb{T}_{\mathcal{I}}) e^{-br^{1/(d+1)}}.
\end{equation}
The constants $C_{\rm apx},b>0$ depend only on the boundary value problem \eqref{eq:modelstrong}, 
$\Omega$, $d$, $p$, and the $\gamma$-shape regularity of $\mathcal{T}_h$.
\end{theorem}
\begin{remark}
Typical clustering strategies such as the ``geometric clustering'' described in \cite{Hackbusch} and applied to 
quasiuniform meshes with $\mathcal{O}(N)$ elements lead to fairly balanced cluster trees $\mathbb{T}_{\mathcal{I}}$ 
of depth $\mathcal{O}(\log N)$ and feature a sparsity constant $C_{\rm sp}$ that is bounded uniformly in $N$.
We refer to~\cite{Hackbusch} for the fact that the
memory requirement to store $\mathbf{B}_{\mathcal{H}}$ is $\mathcal{O}\big((r+n_{\rm leaf}) N \log N\big)$.
\end{remark}

\begin{remark}
With the estimate 
$\frac{1}{\norm{\mathbf{A}^{-1}}_2} \lesssim N^{-1}$ from 
\cite[Theorem 2]{ern-guermond}, 
we get a bound for the relative error
\begin{equation}
\frac{\norm{\mathbf{A}^{-1} - \mathbf{B}_{\mathcal{H}}}_2}{\norm{\mathbf{A}^{-1}}_2}
\lesssim C_{\rm apx} C_{\rm sp}{\rm depth}(\mathbb{T}_{\mathcal{I}}) e^{-br^{1/(d+1)}}.
\end{equation}
\end{remark}

Let us conclude this section with an observation concerning the admissibility
condition \eqref{eq:admissibility}. 
If the operator $L$ is symmetric, i.e. $\mathbf{b}=\mathbf{0}$, then the admissibility condition \eqref{eq:admissibility} 
 can be replaced by the weaker admissibility condition
\been\label{eq:minadmissible}
\min\{{\rm diam}B_{R_{\tau}},{\rm diam}B_{R_{\sigma}}\} \leq  \eta \; {\rm dist}(B_{R_{\tau}},B_{R_{\sigma}}).
\een
This follows from the fact that Proposition~\ref{thm:function-approximation}
only needs an admissibility criterion of the form 
${\rm diam}B_{R_{\tau}} \leq \eta \, {\rm dist}(B_{R_{\tau}},B_{R_{\sigma}})$.
Due to the symmetry of $L$, deriving a block approximation for the block $\tau\times\sigma$
is equivalent to deriving an approximation for the block $\sigma\times\tau$. 
Therefore, we can interchange roles of the boxes $B_{R_{\tau}}$ and $B_{R_{\sigma}}$, 
and as a consequence the
weaker admissibility condition \eqref{eq:minadmissible} is sufficient. We summarize this observation 
in the following corollary.

\bcor
In the symmetric case $\mathbf{b} = \mathbf{0}$, the results from Theorem~\ref{th:blockapprox} and 
Theorem~\ref{th:Happrox} hold verbatim with the weaker admissibility criterion \eqref{eq:minadmissible} instead
of \eqref{eq:admissibility}.
\ecor

\section{Low-dimensional approximation of the Galerkin solution on admissible blocks}
\label{sec:Approximation-solution}

In terms of functions and function spaces, the question of approximating the matrix block
$\mathbf{A}^{-1}|_{\tau\times\sigma}$ by a low-rank factorization $\mathbf{X}_{\tau\sigma}\mathbf{Y}_{\tau\sigma}^T$
 can be rephrased as one of how well one can approximate locally the solution of 
certain variational problems. More precisely, we consider, 
for data $f$ supported by $B_{R_\sigma}\cap\Omega$, 
 the problem to find 
 $\phi_h \in S^{p,1}_0({\mathcal T}_h,\Gamma_D)$ such that
\been \label{eq:model}
a(\phi_h,\psi_h) 
 =\langle f,\psi_h\rangle_{L^2(\Omega)}, 
\qquad 
\forall \psi_h \in S^{p,1}_0({\mathcal T}_h,\Gamma_D) . 
\een
We remark in passing that existence and uniqueness of
$\phi_h$ follow from coercivity of $a(\cdot,\cdot)$. 
The question of approximating the matrix block 
$\mathbf{A}^{-1}|_{\tau\times\sigma}$ by a low-rank factorization 
is intimately linked to the question of approximating $\phi_h|_{B_{R_\tau}\cap\Omega}$ 
from low-dimensional spaces. The latter problem is settled in the affirmative 
in the following proposition for $\eta$-admissible cluster pairs $(\tau,\sigma)$: 

\begin{proposition}\label{thm:function-approximation}
Let $(\tau,\sigma)$ be a cluster pair with bounding boxes $B_{R_\tau}$, $B_{R_\sigma}$. 
Assume $ \eta \dist(B_{R_\tau},B_{R_\sigma}) \geq \diam(B_{R_\tau})$ for some $\eta > 0$. 
Fix $q \in (0,1)$. 
Let $\Pi^{L^2}: L^2(\Omega) \ra S^{p,1}_0(\T_h,\Gamma_D)$ be the $L^2(\Omega)$-orthogonal projection.
Then, for each $k\in\mathbb{N}$ there exists a space
$V_k\subset S^{p,1}_0({\mathcal T}_h,\Gamma_D)$ with $\dim V_k\leq C_{\rm dim} (2+\eta)^d q^{-d}k^{d+1}$ 
such that for arbitrary $f \in L^{2}(\Omega)$ with 
$\supp  f  \subset B_{R_\sigma}\cap\Omega$, the solution $\phi_h$ of \eqref{eq:model} 
satisfies 
\begin{equation}
\label{eq:thm:function-approximation-1}
\min_{v \in V_k} \|\phi_h - v\|_{L^2(B_{R_{\tau}}\cap\Omega)} 
\leq C_{\rm box} q^k \|\Pi^{L^2}f\|_{L^2(\Omega)}
\leq C_{\rm box} q^k \|f\|_{L^2(B_{R_\sigma}\cap\Omega)}.
\end{equation}
The constant $C_{\rm box}>0$ depends only on the boundary value problem \eqref{eq:modelstrong} and $\Omega$, 
while $C_{\rm dim}>0$ additionally depends on $p$, $d$, and the $\gamma$-shape regularity of $\mathcal{T}_h$.
\end{proposition}

The proof of Proposition~\ref{thm:function-approximation} will be given at the end 
of this section. The basic steps are as follows: First, one observes that 
$\supp f \subset B_{R_{\sigma}}\cap \Omega$ together with the admissibility condition 
${\rm dist}(B_{R_{\tau}},B_{R_{\sigma}})\geq \eta^{-1}{\rm diam}(B_{R_{\tau}}) > 0$ imply the orthogonality condition
\been\label{eq:ortho}
a(\phi_h,\psi_h)= \skp{f,\psi_{h}}_{L^2(B_{R_{\sigma}}\cap \Omega) } = 0, \quad
  \forall \psi_h \in S^{p,1}_0({\mathcal T}_h,\Gamma_D) \,\text{with} \,\supp \psi_h \! \subset \! B_{R_{\tau}} \cap \Omega.
\een
Second, this observation will allow us to 
prove a Caccioppoli-type estimate (Lemma~\ref{lem:Caccioppoli}) in which stronger norms of $\phi_h$
are estimated by weaker norms of $\phi_h$ on slightly enlarged regions. Third, we proceed 
as in \cite{BebendorfHackbusch,Boerm} by iterating an approximation result (Lemma~\ref{lem:lowdimapp})
 derived from the Scott-Zhang interpolation of the Galerkin solution $\phi_h$.
This iteration argument accounts for the exponential convergence (Lemma~\ref{cor:lowdimapp}). \newline

\subsection{The space ${\mathcal H}_h(D,\omega)$ and a Caccioppoli type estimate}
It will 
be convenient to introduce, for $\rho \subset \mathcal{I}$, the set  
\been\label{eq:subset}
\omega_{\rho} := {\rm interior}\left(\bigcup_{j\in\rho}\supp \psi_j\right) \subseteq \Omega; 
\een
we will implicitly assume henceforth that such sets are unions of elements. 
Let $D\subset \R^d$ be a bounded open set and $\omega\subset \Omega$ be of the form 
given in (\ref{eq:subset}). 
The orthogonality property that we have identified 
in \eqref{eq:ortho} is captured by the following space ${\mathcal H}_{h}(D,\omega)$: 
\begin{eqnarray}\label{eq:approxspace}
\mathcal{H}_h(D,\omega) &:=& \{u\in H^1(D\cap\omega) \colon 
\exists \widetilde{u} \in S^{p,1}_0(\mathcal{T}_h,\Gamma_D) \;
\mbox{s.t.} \ 
u|_{D\cap\omega} = \widetilde{u}|_{D\cap\omega} , \supp \widetilde{u} \subset \ol{\omega}, \nonumber \\ 
& & \; \phantom{\{u \in }  a(u,\psi_h) = 0, \; \forall \; \psi_h \in S^{p,1}_0(\mathcal{T}_h,\Gamma_D) \, \text{with} \, \supp \psi_h \subset \ol{D\cap\omega}\}.
\end{eqnarray} 
For the proof of Proposition~\ref{thm:function-approximation} 
and subsequently Theorems~\ref{th:blockapprox} and \ref{th:Happrox}, we will only need the special
case $\omega = \Omega$; 
the general case ${\mathcal H}_h(D,\omega)$ with $\omega \ne \Omega$ will be required in our analysis
of $LU$-decompositions in Section~\ref{sec:LU-decomposition}. 

Clearly, the finite dimensional space $\mathcal{H}_{h}(D,\omega)$ 
is a closed subspace of $H^1(D\cap\omega)$, 
and we have $\phi_h \in \mathcal{H}_{h}(B_{R_{\tau}},\Omega)$ for the solution $\phi_h$ of \eqref{eq:model} 
with $\supp f \subset B_{R_{\sigma}}\cap\Omega$ and bounding boxes $B_{R_{\tau}},B_{R_{\sigma}}$ 
that satisfy the $\eta$-admissibility criterion \eqref{eq:admissibility}. 
Since multiplications of elements of ${\mathcal H}_h(D,\omega)$ with cut-off function 
and trivial extensions to $\Omega$ appear repeatedly in the sequel, we note the following 
very simple lemma: 
\begin{lemma}
\label{lemma:eta-u}
Let $\omega$ be a union of elements, $D \subset \R^d$ be bounded and open, and 
$\eta \in W^{1,\infty}(\R^d)$ with $\supp \eta \subset D$. For $u \in {\mathcal H}_h(D,\omega)$ 
define the function $\eta u$ pointwise on $\Omega$ 
by $(\eta u)(x):= \eta(x) u(x)$ for $x \in D\cap \omega$ 
and $(\eta u)(x) = 0$ for $x \not\in D \cap \omega$. Then 
\begin{enumerate}[(i)]
\item 
\label{item:lemma:eta-u-1}
$\eta u \in H^1_0(\Omega;\Gamma_D)$ 
\item 
\label{item:lemma:eta-u-2}
$\supp (\eta u) \subset \overline{D\cap\omega}$
\item 
\label{item:lemma:eta-u-3}
If $\eta \in S^{q,1}({\mathcal T}_h)$,  then $\eta u \in S^{p+q,1}_0({\mathcal T}_h, \Gamma_D)$. 
\end{enumerate}
\end{lemma}
\begin{proof}
We only illustrate (\ref{item:lemma:eta-u-1}). 
Given $u \in {\mathcal H}_h(D,\omega)$ there exists by definition 
a function $\widetilde u \in S^{p,1}_0({\mathcal T}_h,\Gamma_D)$ with 
$\supp \widetilde u \subset \overline{\omega}$. By the support properties of 
$\eta$ and $\widetilde u$, the function $\eta u$ coincides with $\eta \widetilde u$. 
As the product of an $H^1(\Omega)$-function and a Lipschitz continuous function,  
the function $\eta \widetilde u$ is in $H^1(\Omega)$.  
\end{proof}

A main tool in our proofs is a Scott-Zhang projection 
$J_h: H^1_0(\Omega;\Gamma_D) \rightarrow S^{p,1}_0({\mathcal T}_h; \Gamma_D)$ of the form introduced 
in \cite{ScottZhang}. It can be selected to have the following additional mapping property 
for any chosen union $\omega$ of elements: 
\begin{equation}
\label{eq:scott-zhang-bc}
\supp u \subset \overline{\omega} \quad \Longrightarrow \quad \supp J_h u  \subset \overline{\omega}. 
\end{equation}
By $\omega_T := \bigcup\left\{T' \in \mathcal{T}_h \;:\; T \cap T' \neq \emptyset\right\}$, we denote the element patch of $T$, 
which contains $T$ and all elements $T' \in \mathcal{T}_h$ that have a common node with $T$. Then, $J_h$ has
the following local approximation property for $\mathcal{T}_h$-piecewise $H^{\ell}$-functions 
$u \in H^{\ell}_{\text{pw}}(\T_h,\omega) := \{u\in L^2(\omega) : u|_{T} \in H^{\ell}(T)\, \forall \, T \in \T_h\}$
\begin{equation}\label{eq:SZapprox}
\norm{u-J_h u}_{H^m(T)}^2 \leq C h^{2(\ell-m)}\! \sum_{T' \subset \omega_T}\abs{u}_{H^{\ell}(T')}^2, \; 0 \leq m \leq 1, \ m \leq \ell \leq p+1.
\end{equation}
The constant $C>0$ depends only on $\gamma$-shape regularity of $\mathcal{T}_h$, the dimension $d$, and the polynomial degree $p$. In particular, it is independent of the choice of the set $\omega$ in (\ref{eq:scott-zhang-bc}).

In the following, we will construct approximations on nested boxes and therefore 
introduce the notion of concentric boxes.
\begin{definition}[concentric boxes]
 Two boxes $B_R$, $B_{R'}$ of side length $R,R'$ are said to be concentric, if they have the same barycenter and
$B_R$ can be obtained by a stretching of $B_{R'}$ by the factor $R/R'$ taking their common barycenter as the origin. 
\end{definition}

For a box $B_R$ with side length $R \leq 2 \,{\rm diam}(\Omega)$, we introduce the norm
\begin{equation*}
\triplenorm{u}_{h,R}^2 := \left(\frac{h}{R}\right)^2 \norm{\nabla u}^2_{L^2(B_R\cap\omega)} + \frac{1}{R^2}\norm{u}_{L^2(B_R\cap\omega)}^2,
\end{equation*}
which is, for fixed $h$, equivalent to the $H^1$-norm. 
The following lemma states a Caccioppoli-type estimate for functions in $\mathcal{H}_h(B_{(1+\delta)R},\omega)$, 
where $B_{(1+\delta)R}$ and $B_R$ are concentric boxes.

\begin{lemma}\label{lem:Caccioppoli}
Let $\delta\in (0,1)$, $\frac{h}{R} \leq \frac{\delta}{4}$ and let $\omega \subseteq \Omega$ 
be of the form \eqref{eq:subset}. Let $B_{R}$, $B_{(1+\delta)R}$ be two concentric boxes. 
Let $u\in\mathcal{H}_h(B_{(1+\delta)R},\omega)$.  
Then, there exists a constant $C_{\text{reg}} > 0$ which depends only on 
the boundary value problem \eqref{eq:modelstrong},
 $\Omega$, $d$, $p$, and the $\gamma$-shape regularity of $\T_h$, such that 
\begin{equation}\label{eq:Caccioppoliiterated}
\norm{\nabla u}_{L^2(B_{R}\cap\omega)} \leq \norm{\nabla u}_{L^2(B_{R}\cap\omega)} + \skp{\alpha u,u}_{L^2(B_R \cap (\Gamma_{\mathcal{R}}\cap\ol{\omega}))}^{1/2} \leq C_{\text{reg}}\frac{1+\delta}{\delta} \triplenorm{u}_{h,(1+\delta)R}.
\end{equation}
\end{lemma}

\begin{proof}
Let $\eta \in S^{1,1}(\T_h)$ be a piecewise affine cut-off function with 
$\supp \eta \subset B_{(1+\delta/2)R}\cap\overline{\Omega}$, 
$\eta \equiv 1$ on $B_{R}\cap\omega$, $0\leq\eta\leq 1$,
and $\norm{\nabla \eta}_{L^{\infty}(B_{(1+\delta)R}\cap\Omega)} \lesssim \frac{1}{\delta R}$, 
$\norm{D^2 \eta}_{L^{\infty}(B_{(1+\delta)R}\cap\Omega)} \lesssim \frac{1}{\delta^2R^2}$.
By Lemma~\ref{lemma:eta-u} we have  
$\eta^2 u \in S^{p+2,1}_0({\mathcal T}_h,\Gamma_D) \subset H^1_0(\Omega;\Gamma_D)$ and
\begin{equation}
\label{eq:support-eta^2u}
\supp (\eta^2 u) \subset \overline{B_{(1+\delta/2)R} \cap \omega}. 
\end{equation}
Recall that $h$ is the maximal element diameter and $4h\leq \delta R$. Hence, for the 
Scott-Zhang operator $J_h$, we have $\supp J_h (\eta^2 u) \subset \overline{B_{(1+\delta) R}}$; 
in view of (\ref{eq:scott-zhang-bc}) we furthermore have $\supp J_h (\eta^2 u) \subset \overline{\omega}$
so that 
\begin{equation}
\label{eq:support-J_heta^2u}
\supp J_h (\eta^2 u) \subset \overline{B} \quad \text{with} \,  B:= B_{(1+\delta) R} \cap \omega.
\end{equation}
With the coercivity of the bilinear form $a(\cdot,\cdot)$ and $\frac{1}{\delta R} \lesssim \frac{1}{\delta^2R^2}$,
 since $\delta < 1$ and $R \leq 2\diam(\Omega)$, we have
\begin{subequations}
\begin{align}
\norm{\nabla u}_{L^2(B_R \cap \omega)}^2 +\skp{\alpha u,u}_{L^2(\overline{B_{R}\cap\omega}\cap\Gamma_R)}
 &\leq \norm{\nabla(\eta u)}_{L^2(B)}^2 + \skp{\alpha \eta u, \eta u}_{L^2(\overline{B}\cap\Gamma_R)} \label{eq:CacTemp1a}\\
&\lesssim a(\eta u,\eta u) \nonumber \\
&=  \int_{B}{\mathbf{C} \nabla u \cdot \nabla (\eta^2 u) + u^2  \mathbf{C} \nabla \eta \cdot \nabla \eta \, dx}  
+ \skp{\mathbf{b}\cdot \nabla u+\beta u, \eta^2 u}_{L^2(B)} +\nonumber \\
& \quad \skp{\mathbf{b}\cdot (\nabla \eta) u,\eta u}_{L^2(B)} +
\skp{\alpha u, \eta^2 u}_{L^2(\overline{B} \cap \Gamma_{\mathcal{R}})} +  \frac{1}{\delta^2 R^2}\norm{u}_{L^2(B)}^2\nonumber\\
&\lesssim   \int_{B}{\mathbf{C} \nabla u \cdot \nabla (\eta^2 u) dx}  
+ \skp{\mathbf{b}\cdot \nabla u+\beta u, \eta^2 u}_{L^2(B)} +\nonumber \\
&  \quad
\skp{\alpha u, \eta^2 u}_{L^2(\overline{B} \cap \Gamma_{\mathcal{R}})} +  \frac{1}{\delta^2 R^2}\norm{u}_{L^2(B)}^2\nonumber \\
&= a(u,\eta^2 u) + \frac{1}{\delta^2 R^2}\norm{u}_{L^2(B)}^2.\label{eq:CacTemp1b}
\end{align}
\end{subequations}
Recall from (\ref{eq:support-J_heta^2u}) that 
$\supp J_h (\eta^2 u) \subset \overline{B}$. 
The orthogonality relation \eqref{eq:approxspace} in the definition of the space $\H_h(B,\omega)$
therefore implies
\begin{eqnarray}\label{eq:tempCaccioppoli}
a(u,\eta^2 u) &=& a(u,\eta^2u - J_h(\eta^2 u))\nonumber \\
  &\leq& \norm{\mathbf{C}}_{L^{\infty}(B)} \norm{\nabla u}_{L^2(B)}\norm{\nabla(\eta^2u - J_h(\eta^2 u))}_{L^2(B)} \nonumber\\
& &+ \left(\norm{\mathbf{b}}_{L^{\infty}(B)} \norm{\nabla u}_{L^2(B)}+\norm{\beta}_{L^{\infty}(B)}\norm{\eta u}_{L^2(B)}\right)\norm{\eta^2u - J_h(\eta^2 u)}_{L^2(B)}  \nonumber \\
& & +\abs{\skp{\alpha u,\eta^2 u -J_h(\eta^2 u)}_{L^2(\overline{B}\cap\Gamma_{\mathcal{R}})}}.
\end{eqnarray}
The approximation property \eqref{eq:SZapprox}, the requirement (\ref{eq:scott-zhang-bc}), 
and the support properties of $\eta^2 u$ lead to
\begin{equation}
\label{eq:foo}
\norm{\nabla(\eta^2u - J_h(\eta^2 u))}_{L^2(\Omega)}^2  \lesssim h^{2p}\sum_{\substack{T \in \mathcal{T}_h \\ T \subseteq B}}\norm{D^{p+1}(\eta^2 u)}_{L^2(T)}^2.
\end{equation}
Since, for each $T \subset B$ we have  $u|_T \in {\mathcal P}_p$,  we get $D^{p+1}u|_T = 0$ and $\eta \in S^{1,1}(\T_h)$ implies $D^j \eta|_T = 0$ for $j \geq 2$.
With the Leibniz product rule, the right-hand side of (\ref{eq:foo}) can therefore be estimated by
\begin{eqnarray*}
\norm{D^{p+1}(\eta^2 u)}_{L^2(T)}  &\lesssim& \norm{D^2 \eta^2 D^{p-1}u + \eta \nabla \eta D^{p} u}_{L^2(T)} \lesssim
\norm{\nabla \eta \cdot \nabla \eta D^{p-1}u + \eta \nabla \eta D^{p} u}_{L^2(T)} \\
&\lesssim& \frac{1}{\delta R}\norm{\nabla \eta D^{p-1}u + \eta D^{p} u}_{L^2(T)}
\lesssim \frac{1}{\delta R} \norm{D^{p}(\eta u)}_{L^2(T)},
\end{eqnarray*} 
where the suppressed constant depends on $p$.
The inverse inequality
$\displaystyle 
\norm{D^p(\eta u)}_{L^2(T)} \lesssim h^{-p+1}\norm{\nabla(\eta u)}_{L^2(T)}
$, see e.g. \cite{GHS}, 
leads to
\begin{eqnarray}\label{eq:tempCaccioppoli2}
\norm{\nabla(\eta^2u - J_h(\eta^2 u))}_{L^2(\Omega)}^2  &\lesssim&
 \frac{1}{\delta^2 R^2}h^{2p } \sum_{\substack{T \in \mathcal{T}_h \\ T \subseteq B}}\norm{D^{p}(\eta u)}_{L^2(T)}^2 
\lesssim \frac{h^2}{\delta^2R^2} \norm{\nabla(\eta u)}_{L^2(B)}^2 \nonumber \\
&\lesssim& \frac{h^2}{\delta^4R^4} \norm{u}_{L^2(B)}^2 + \frac{h^2}{\delta^2R^2} \norm{\eta \nabla u}_{L^2(B)}^2.
\end{eqnarray}
The same line of reasoning leads to
\begin{equation}\label{eq:tempCaccioppoli3}
\norm{\eta^2u - J_h(\eta^2 u)}_{L^2(\Omega)}\lesssim \frac{h^2}{\delta^2R^2} \norm{u}_{L^2(B)} + \frac{h^2}{\delta R} \norm{\eta \nabla u}_{L^2(B)}.
\end{equation}
In order to derive an estimate for the boundary term in \eqref{eq:tempCaccioppoli}, 
we need a second smooth cut-off function $\widetilde \eta$ 
with $\supp \widetilde \eta \subset \overline{B_{(1+\delta)R}}$ and 
$\widetilde \eta \equiv 1$ on $\supp (J_h(\eta^2 u) - \eta^2 u)$ and 
$\|\nabla \widetilde \eta\|_{L^\infty(B_{(1+\delta)R})} \lesssim \frac{1}{\delta R}$. 
By Lemma~\ref{lemma:eta-u} we can define the function 
$\widetilde \eta u \in H^1(\Omega)$ with the support property 
$\displaystyle 
\supp \widetilde \eta u \subset \overline{B_{(1+\delta)R} \cap \omega } = \overline{B} 
$
and therefore 
\begin{equation}
\label{eq:foo-1}
\|\widetilde \eta u\|_{H^1(\Omega)} \leq \norm{u}_{L^2(B)}+\|\nabla(\widetilde \eta u)\|_{L^2(B)} \lesssim \frac{1}{\delta R} \|u\|_{L^2(B)} + \|\nabla u\|_{L^2(B)}. 
\end{equation}
Then, we get
\bea
\abs{\skp{\alpha u,\eta^2 u -J_h(\eta^2 u)}_{L^2(\overline{B}\cap\Gamma_{\mathcal{R}})}} &=& \abs{\skp{\alpha \widetilde{\eta} u,\eta^2 u -J_h(\eta^2 u)}_{L^2(\overline{B}\cap\Gamma_{\mathcal{R}})}}\\
&\leq& \norm{\alpha}_{L^{\infty}(\overline{B}\cap\Gamma_{\mathcal{R}})}\norm{\widetilde{\eta} u}_{L^2(\overline{B}\cap\Gamma_{\mathcal{R}})}\norm{\eta^2 u \!- \!J_h(\eta^2 u)}_{L^2(\overline{B}\cap\Gamma_{\mathcal{R}})} . 
\eea
The multiplicative trace inequality for $\Omega$ and the estimate 
\eqref{eq:foo-1} gives 
\begin{align*}
 \|\widetilde \eta u\|_{L^2(\Gamma)}  \lesssim  
\|\widetilde \eta u\|^{1/2}_{L^2(\Omega)} \|\widetilde \eta u\|^{1/2}_{H^1(\Omega)} 
\lesssim \frac{1}{\sqrt{\delta R}} \|u\|_{L^2(B)} + \|u\|^{1/2}_{L^2(B)} \|\nabla u\|^{1/2}_{L^2(B)}.
\end{align*}
The multiplicative trace inequality for $\Omega$ 
and the estimates \eqref{eq:tempCaccioppoli2} -- \eqref{eq:tempCaccioppoli3} imply
\begin{align*}
& \|\eta^2 u - J_h (\eta^2 u)\|_{L^2(\Gamma)} \lesssim 
\|\eta^2 u - J_h (\eta^2 u)\|_{L^2(\Omega)} + 
\|\eta^2 u - J_h (\eta^2 u)\|_{L^2(\Omega)}^{1/2}  \|\nabla( \eta^2 u - J_h (\eta^2 u))\|_{L^2(\Omega)}^{1/2}  \\
&\lesssim  \left( \frac{h^2}{\delta^2 R^2} \|u\|_{L^2(B)} + \frac{h^2}{\delta R}\|\nabla u\|_{L^2(B)} \right)  + 
\left( \frac{h}{\delta R} \|u\|^{1/2}_{L^2(B)} + \frac{h}{\sqrt{\delta R}}\|\nabla u\|^{1/2}_{L^2(B)} \right) 
\left( \frac{\sqrt{h}}{\delta R} \|u\|^{1/2}_{L^2(B)} + \frac{\sqrt{h}}{\sqrt{\delta R}}\|\nabla u\|^{1/2}_{L^2(B)} \right) \\
&\lesssim \frac{h^{3/2}}{(\delta R)^2} \|u\|_{L^2(B)}  + \frac{h^{3/2}}{\delta R} \|\nabla u\|_{L^2(B)} 
+ \frac{h^{3/2}}{(\delta R)^{3/2}} \|u\|^{1/2}_{L^2(B)} \|\nabla u\|^{1/2}_{L^2(B)} \\
&\lesssim \frac{h^{3/2}}{(\delta R)^2} \|u\|_{L^2(B)}  + \frac{h^{3/2}}{\delta R} \|\nabla u\|_{L^2(B)}. 
\end{align*}
Therefore, 
\begin{align*}
& \norm{\widetilde{\eta} u}_{L^2(\Gamma)} \norm{\eta^2 u -J_h(\eta^2 u)}_{L^2(\Gamma)} \lesssim 
\left(\frac{1}{\sqrt{\delta R}} \|u\|_{L^2(B)} + \|u\|^{1/2}_{L^2(B)}\|\nabla u\|^{1/2}_{L^2(B)} 
\right) 
\left( \frac{h^{3/2}}{(\delta R)^2} \|u\|_{L^2(B)} + \frac{h^{3/2}}{\delta R} \|\nabla u\|_{L^2(B)} \right)
\\
& \lesssim 
\frac{h^{3/2}}{(\delta R)^{5/2}} \|u\|^2_{L^2(B)} 
+ 
\frac{h^{3/2}}{(\delta R)^{3/2}} \|u\|_{L^2(B)} \|\nabla u\|_{L^2(B)} 
+ 
\frac{h^{3/2}}{(\delta R)^{2}} \|u\|^{3/2}_{L^2(B)} \|\nabla u\|^{1/2}_{L^2(B)} 
+ 
\frac{h^{3/2}}{\delta R} \|u\|^{1/2}_{L^2(B)} \|\nabla u\|^{3/2}_{L^2(B)}. 
\end{align*}
Young's inequality and $h/(\delta R) \leq 1/4$ allow us to conclude 
(rather generously) 
\bean\label{eq:tempCaccioppoli4}
\abs{\skp{\alpha u,\eta^2 u -J_h(\eta^2 u)}_{L^2(\overline{B}\cap\Gamma_{\mathcal{R}})}} &\lesssim& \norm{\widetilde{\eta} u}_{L^2(\Gamma)} \norm{\eta^2 u -J_h(\eta^2 u)}_{L^2(\Gamma)}\nonumber \\
&\lesssim& \! \!
\frac{h^2}{(\delta R)^2} \|\nabla u\|^2_{L^2(B)}\! +\! \frac{1}{(\delta R)^2} \|u\|^2_{L^2(B)}
\! = \! \left(\frac{1+\delta}{\delta}\right)^2\!\! \triplenorm{u}^2_{h,(1+\delta)R}. 
\eean
Inserting the estimates \eqref{eq:tempCaccioppoli2}, \eqref{eq:tempCaccioppoli3}, \eqref{eq:tempCaccioppoli4} 
into \eqref{eq:tempCaccioppoli} and with Young's inequality, we get with \eqref{eq:CacTemp1b} that
\begin{eqnarray*}
\norm{\nabla(\eta u)}^2_{L^2(B)} + \skp{\alpha\eta u, \eta u}_{L^2(\overline{B}\cap\Gamma_{\mathcal{R}})} &\lesssim&
 a(u,\eta^2 u) +   \frac{1}{\delta^2 R^2}\norm{u}_{L^2(B)}^2  \\
&\lesssim& \norm{\nabla u}_{L^2(B)}\left(\frac{h}{\delta^2R^2}\norm{u}_{L^2(B)}+\frac{h}{\delta R}\norm{\eta\nabla u}_{L^2(B)}\right) \\
& &+ \left(\norm{\nabla u}_{L^2(B)}+\norm{\eta u}_{L^2(B)}\right)\left(\frac{h^2}{\delta^2R^2}\norm{u}_{L^2(B)}+\frac{h^2}{\delta R}\norm{\eta\nabla u}_{L^2(B)}\right)  \\
& & +\frac{h^{2}}{\delta^2 R^{2}}\norm{\nabla u}_{L^2(B)}^2 + \frac{1}{\delta^2 R^2}\norm{u}_{L^2(B)}^2  \\
&\leq& C(\epsilon)\frac{h^2}{\delta^2R^2}\norm{\nabla u}_{L^2(B)}^2 
 +C(\epsilon)\frac{1}{\delta^2R^2}\norm{u}_{L^2(B)}^2+ \epsilon\norm{\eta\nabla u}_{L^2(B)}^2. 
\end{eqnarray*}
Moving the term $\epsilon \norm{\eta \nabla u}_{L^2(B)}^2$ to the left-hand side and 
inserting this estimate in \eqref{eq:CacTemp1a}, we conclude the proof.
\end{proof}

\subsection{Low-dimensional approximation in ${\mathcal{H}_{h}}(D,\omega)$}
In this subsection, we will derive a low dimensional approximation of the Galerkin solution
by Scott-Zhang interpolation on a coarser grid.

%
We need to be able to extend functions defined on $B_{(1+2\delta)R}\cap\omega$ to ${\mathbb R}^d$. 
To this end, we use an extension operator $E: H^1(\Omega) \ra H^1(\R^d)$, 
see e.g. \cite[Theorem 4.32]{Adams}, which satisfies $E u = u$ on $\Omega$ and the $H^1$-stability estimate
\bee
\norm{{E} u}_{H^1(\R^d)} \leq C \norm{u}_{H^1(\Omega)}.
\ee
For a function $u \in \H_h(B_{(1+2\delta)R},\omega)$ and a cut-off function 
$\eta \in C_0^{\infty}(B_{(1+2\delta)R})$ with 
$\supp \eta \subset B_{(1+\delta)R}$, $\eta \equiv 1$ on $B_{R}\cap\omega$ we can define the function 
$\eta u \in H^1(\Omega)$ with the aid of Lemma~\ref{lemma:eta-u}. We note 
the support property $\supp (\eta u) \subset \overline{B_{(1+2\delta)R} \cap \omega}$, 
due to $\supp u \subset \overline{\omega}$. Therefore, the extension 
of $\eta u$ to $\Omega$ by zero is in $H^1(\Omega)$. Therefore, we have 
\been\label{eq:extension}
\norm{E(\eta u)}_{H^1(\R^d)} \leq C \norm{\eta u}_{H^1(\omega)}.
\een
Moreover, let 
$\Pi_{h,R} : (H^1(B_R\cap\omega),\triplenorm{\cdot}_{h,R}) \rightarrow (\mathcal{H}_h(B_R,\omega), \triplenorm{\cdot}_{h,R})$ 
be the orthogonal projection, 
which is well-defined since $\mathcal{H}_h(B_R,\omega)\subset H^1(B_R\cap\omega)$ is a closed subspace. 

\begin{lemma}\label{lem:lowdimapp}
Let $\delta\in(0,1)$, $B_R$, $B_{(1+\delta)R}$, and $B_{(1+2\delta)R}$ concentric boxes,
$\omega \subseteq \Omega$ of the form \eqref{eq:subset} 
and $u\in\mathcal{H}_h(B_{(1+2\delta)R},\omega)$. 
Assume $\frac{h}{R}\leq \frac{\delta}{4}$.
Let $\mathcal{K}_H $ be an (infinite) $\gamma$-shape regular triangulation of $\mathbb{R}^d$ and assume 
$\frac{H}{R} \leq \frac{\delta}{4}$ for the corresponding mesh width $H$.
Let $\eta \in C_0^{\infty}(B_{(1+2\delta)R})$ be a cut-off function satisfying 
$\supp \eta \subset B_{(1+\delta)R}$, $\eta \equiv 1$ on $B_{R}\cap\omega$, and 
$\norm{\nabla \eta}_{L^{\infty}(B_{(1+2\delta)R})} \lesssim \frac{1}{\delta R}$.
Moreover, let $J_H: H^1(\mathbb{R}^d) \rightarrow S^{p,1}(\mathcal{K}_H)$ be the Scott-Zhang projection
and $E:H^1(\Omega) \rightarrow H^1(\R^d)$ be an $H^1$-stable extension operator. 
Then, there exists a constant $C_{\rm app} > 0$, which depends only on the boundary value problem 
\eqref{eq:modelstrong}, $\Omega$, $d$, $p$, $\gamma$, and $E$ such that
\begin{enumerate}[(i)]
\item 
\label{item:lem:lowdimapp-ii}
$\big(u-\Pi_{h,R}J_H E (\eta u)\big)|_{B_{R}\cap\omega} \in \mathcal{H}_h(B_{R},\omega)$ \\[-2mm]
\item 
\label{item:lem:lowdimapp-i}
$\triplenorm{u-\Pi_{h,R}J_H E(\eta u)}_{h,R} \leq C_{\rm app}\frac{1+2\delta}{\delta} \left(\frac{h}{R}+\frac{H}{R}\right)\triplenorm{u}_{h,(1+2\delta)R}$ 
\item 
\label{item:lem:lowdimapp-iii}
$\dim W\leq C_{\rm app}\left(\frac{(1+2\delta)R}{H}\right)^d$, where $W:=\Pi_{h,R}J_H E \mathcal{H}_h(B_{(1+2\delta)R},\omega) $. 
\end{enumerate}
\end{lemma}

\begin{proof}
The statement (\ref{item:lem:lowdimapp-iii}) follows from the fact that 
$\dim J_H(E \mathcal{H}_h(B_{(1+2\delta)R},\omega))|_{B_{(1+\delta)R}} \lesssim ((1+2\delta)R/H)^d$. 
For $u \in \mathcal{H}_h(B_{(1+2\delta)R},\omega)$, 
we have $u|_{B_R \cap \omega} \in \mathcal{H}_h(B_{R},\omega)$ as well and hence 
$\Pi_{h,R}\left(u|_{B_{R}\cap\omega}\right) = u|_{B_{R}\cap\omega}$, which gives (\ref{item:lem:lowdimapp-ii}). 
It remains to prove (\ref{item:lem:lowdimapp-i}):
The assumption $\frac{H}{R} \leq \frac{\delta}{4}$ implies 
$\bigcup\{K \in \mathcal{K}_H \;:\; \omega_K \cap B_{R} \neq \emptyset\} \subseteq B_{(1+\delta)R}$.
The locality and the approximation properties \eqref{eq:SZapprox} of $J_H$ yield
\begin{eqnarray*}
\frac{1}{H} \norm{E(\eta u) - J_H E (\eta u)}_{L^2(B_{R})} + 
\norm{\nabla(E(\eta u) - J_HE (\eta u))}_{L^2(B_{R})} \! &\lesssim& \!  \norm{\nabla E (\eta u)}_{L^2(B_{(1+\delta)R})}.
\end{eqnarray*}
We apply Lemma~\ref{lem:Caccioppoli} with
$\widetilde{R} = (1+\delta)R$ and $\widetilde{\delta} = \frac{\delta}{1+\delta}$. 
Note that $(1+\widetilde{\delta})\widetilde{R} = (1+2\delta)R$, and
 $\frac{h}{\widetilde{R}}\leq \frac{\widetilde{\delta}}{4}$
follows from $4h \leq \delta R = \widetilde{\delta}\widetilde{R}$. Hence, we obtain with \eqref{eq:extension}
\begin{align*}
&\triplenorm{u-\Pi_{h,R}J_HE (\eta u)}_{h,R}^2 =\triplenorm{\Pi_{h,R}\left(E (\eta u)-J_H E (\eta u)\right)}^2_{h,R} \leq \triplenorm{E (\eta u)-J_H E (\eta u)}_{h,R}^2 \\
& \qquad= \left(\frac{h}{R}\right)^{2}\norm{\nabla (E(\eta u)-J_H E(\eta u))}_{L^2(B_{R}\cap\omega)}^2   + \frac{1}{R^2} \norm{E(\eta u)-J_H E(\eta u)}_{L^2(B_{R}\cap\omega)}^2\\
&\qquad\lesssim\frac{h^2}{R^2}\norm{\nabla E (\eta u)}_{L^{2}(B_{(1+\delta) R})}^2 + \frac{H^2}{R^2}\norm{\nabla E (\eta u)}_{L^2(B_{(1+\delta)R})}^2
\lesssim\left(\frac{h^2}{R^2} + \frac{H^2}{R^2}\right)\norm{\eta u}_{H^1(\Omega)}^2  \\
&\qquad \lesssim 
\left(\frac{h^2}{R^2} + \frac{H^2}{R^2}\right)\frac{1}{\delta^2 R^2}\norm{u}_{L^2(B_{(1+\delta)R} \cap \omega)}^2 + 
\left(\frac{h^2}{R^2} + \frac{H^2}{R^2}\right)\norm{\nabla u}_{L^2(B_{(1+\delta)R} \cap \omega)}^2 \\
 &\qquad\leq  \left(C_{\rm app}\frac{1+2\delta}{\delta}\left(\frac{h}{R}+\frac{H}{R}\right)\right)^2\triplenorm{u}^2_{h,(1+2\delta)R},
\end{align*}
which concludes the proof.
 \end{proof}

By iterating this approximation result on suitable concentric boxes, we can derive a low-dimensional
subspace in the space $\H_h$ and the bestapproximation in this space converges exponentially, 
which is stated in the following lemma. 

\begin{lemma}\label{cor:lowdimapp}
Let $C_{\rm app}$ be the constant of Lemma~\ref{lem:lowdimapp}.
Let $q,\kappa,R \in (0,1)$, $k \in \mathbb{N}$ 
and $\omega \subseteq \Omega$ of the form \eqref{eq:subset}. Assume 
\begin{equation}
\label{eq:cor:lowdimapp-1}
\frac{h}{R} \leq \frac{\kappa q} {8 k \max\{1,C_{\rm app}\}}.
\end{equation}
Then, there exists a subspace $V_k$ of $S^{p,1}_0(\T_h,\Gamma_D)|_{B_R \cap \omega}$ 
with dimension 
$$
\dim V_k \leq C_{\rm dim} \left(\frac{1 + \kappa^{-1}}{q}\right)^dk^{d+1},
$$
such that for every $u \in \mathcal{H}_{h}(B_{(1+\kappa)R},\omega)$  
\begin{equation}\label{eq:lowdimapp}
\min_{v\in V_k} \triplenorm{u-v}_{h,R} \leq q^{k} \triplenorm{u}_{h,(1+\kappa)R}.
\end{equation}
The constant $C_{\rm dim}>0$ depends only on the boundary value problem \eqref{eq:modelstrong},
 $\Omega,d$, and the $\gamma$-shape regularity of $\mathcal{T}_h$.
\end{lemma}

\begin{proof}
We iterate the approximation result of Lemma~\ref{lem:lowdimapp} on boxes $B_{(1+\delta_j)R}$, 
with $\delta_j := \kappa\frac{k-j}{k}$ for $j=0,\dots,k$. 
We note that $\kappa = \delta_0>\delta_1 > \dots >\delta_k = 0$.
We choose $H = \frac{\kappa q R}{8k\max\{C_{\rm app},1\}}$.

 If $h \ge H$, then we select $V_k = \mathcal{H}_h(B_R,\omega)$. 
Due to the choice of $H$ we have 
$\dim V_k \lesssim \left(\frac{R}{h}\right)^d \lesssim k\left(\frac{R}{H}\right)^d \simeq C_{\rm dim} \left(\frac{1 + \kappa^{-1}}{q}\right)^dk^{d+1}$. 

If $h < H$, we apply Lemma~\ref{lem:lowdimapp} with $\widetilde{R}\! =\! (1+\delta_j)R$ and 
$\widetilde{\delta}_j\! =\! \frac{1}{2k(1+\delta_j)}\!<\!\frac{1}{2}$.
Note that $\delta_{j-1} = \delta_j +\frac{1}{k}$ gives $(1+\delta_{j-1})R=(1+2\widetilde{\delta}_j)\widetilde{R}$. 
The assumption $\frac{H}{\widetilde{R}}\leq \frac{1}{8k(1+\delta_j)} = \frac{\widetilde{\delta}_j}{4}$ 
is fulfilled due to our choice of $H$.
For $j=1$, Lemma~\ref{lem:lowdimapp} provides an approximation $w_1$ in a subspace $W_1$ of $\mathcal{H}_h(B_{(1+\delta_1)R},\omega)$ 
with $\dim W_1 \leq C\left(\frac{(1+\kappa)R}{H}\right)^d$ such that
\begin{eqnarray*}
\triplenorm{u-w_1}_{h,(1+\delta_1)R} &\leq& 2C_{\rm app}\frac{H}{(1+\delta_1)R}{\frac{1+2\widetilde{\delta}_1}{\widetilde{\delta}_1}} \triplenorm{u}_{h,(1+\delta_0)R}\\
 &=&4C_{\rm app}\frac{kH}{R}(1+2\widetilde{\delta}_1) \triplenorm{u}_{h,(1+\kappa)R} \leq  q\triplenorm{u}_{h,(1+\kappa)R}.
\end{eqnarray*}
Since $u-w_1 \in \mathcal{H}_h(B_{(1+\delta_1)R},\omega)$, we can use Lemma~\ref{lem:lowdimapp} again 
and get an approximation $w_2$ of $u-w_1$ in a subspace $W_2$ of $\mathcal{H}_h(B_{(1+\delta_2)R},\omega)$ with
$\dim W_2\leq C\left(\frac{(1+\kappa)R}{H}\right)^d$. Arguing as for $j=1$, we get
\begin{equation*}
\triplenorm{u-w_1-w_2}_{h,(1+\delta_2)R} \leq q \triplenorm{u-w_1}_{h,(1+\delta_1)R} \leq  
q^2 \triplenorm{u}_{h,(1+\kappa)R}.
\end{equation*}
Continuing this process $k-2$ times leads to an approximation $v := \sum_{j=1}^kw_i$ 
in the space $V_k := \sum_{j=1}^{k}W_j$ 
of dimension $\dim V_k \leq Ck\left(\frac{(1+\kappa)R}{H}\right)^d=C_{\rm dim} \left(\frac{1 + \kappa^{-1}}{q}\right)^dk^{d+1}$.
 \end{proof}

Now we are able to prove the main result of this section.\\

\begin{proof}[of Proposition~\ref{thm:function-approximation}]
Choose $\kappa = \frac{1}{1+\eta}$. By assumption, we have 
$\dist(B_{R_\tau},B_{R_\sigma}) 
\ge \eta^{-1} \diam B_{R_\tau} = \sqrt{d} \eta^{-1} R_{\tau}$. 
In particular, this implies 
\begin{equation*}
\dist(B_{(1+\kappa) R_\tau},B_{R_\sigma}) \geq \dist(B_{R_\tau},B_{R_\sigma}) - \kappa R_{\tau} \sqrt{d} \geq \sqrt{d}R_{\tau}(\eta^{-1}-\kappa) = \sqrt{d}R_{\tau}\frac{1}{\eta(1+\eta)} >0.
\end{equation*}
The Galerkin solution $\phi_h$ satisfies $\phi_h|_{B_{(1+\delta)R}\cap \Omega} \in \H_h(B_{(1+\delta)R},\Omega)$.
The coercivity \eqref{eq:coercivity} of the bilinear form $a(\cdot,\cdot)$ implies
\begin{eqnarray*}
\norm{\phi_h}_{H^{1}(\Omega)}^2 &\lesssim& a(\phi_h,\phi_h) = \skp{f,\phi_h} = \skp{\Pi^{L^2} f,\phi_h} 
\lesssim \norm{\Pi^{L^2} f}_{L^2(\Omega)}\norm{\phi_h}_{H^1(\Omega)}.
\end{eqnarray*}
Furthermore, with $\frac{h}{R_{\tau}}<1$, we get 
\begin{eqnarray*}
\triplenorm{\phi_h}_{h,(1+\kappa)R_{\tau}} &\lesssim&   \left(1+\frac{1}{R_{\tau}}\right)\norm{\phi_h}_{H^1(\Omega)}
 \lesssim \left(1+\frac{1}{R_{\tau}}\right)\norm{\Pi^{L^2} f}_{L^2(\Omega)} ,
\end{eqnarray*}
and we have a bound on the right-hand side of \eqref{eq:lowdimapp}.
We are now in the position to define the space $V_k$, for which we distinguish two cases. 
\newline
{\bf Case 1:} The condition \eqref{eq:cor:lowdimapp-1} is satisfied with $R = R_{\tau}$. 
With the space $V_k$ provided by Lemma~\ref{cor:lowdimapp} we get
 \begin{eqnarray*}
\min_{v\in V_k}\norm{\phi_h-v}_{L^2(B_{R_{\tau}}\cap\Omega)} &\leq& R_{\tau}\min_{v\in V_k}\triplenorm{\phi_h-v}_{h,R_{\tau}} 
\lesssim (R_{\tau}+1)q^k\norm{\Pi^{L^2}f}_{L^2(\Omega)} \\
 &\lesssim& {\rm diam}(\Omega)q^k\norm{\Pi^{L^2}f}_{L^2(\Omega)},
\end{eqnarray*}
and the dimension of $V_k$ is bounded by $\dim V_k \leq C \left((2 + \eta) q^{-1}\right)^d k^{d+1}.$
\newline
{\bf Case 2:} The condition \eqref{eq:cor:lowdimapp-1} is not satisfied. 
Then, $\frac{h}{R_{\tau}} \geq \frac{\kappa q} {8 k \max\{1,C_{\rm app}\}}$ and we select 
$V_k:= \left\{v|_{B_{R_{\tau}}\cap\Omega} : v \in S^{p,1}_0({\mathcal T}_h,\Gamma_D) \right\}$. 
Then the minimum in \eqref{eq:thm:function-approximation-1} is obviously zero. By choice of $\kappa$, the dimension
of $V_k$ is bounded by 
$$
\dim V_k \lesssim \left(\frac{R_{\tau}}{h}\right)^{d} 
\lesssim \left(\frac{8 k \max\{C_{\rm app},1\}}{\kappa q}\right)^{d}  
\lesssim \left((1+\eta) q^{-1} \right)^{d} k^{d+1}, 
$$
which concludes the proof of the non trivial statement in \eqref{eq:thm:function-approximation-1}.
The other estimate follows directly from the $L^2(\Omega)$-stability of the $L^2(\Omega)$-orthogonal projection.
\end{proof}

\section{The Neumann Problem}
Our techniques employed in the previous chapter can be used to treat 
problems with purely Neumann boundary conditions as well. Our model problem in this case reads in the strong form as 
\bea
Lu:= -\div(\mathbf{C}\nabla u) &=& f \quad \text{in} \; \Omega, \\
\mathbf{C}\nabla u\cdot n &=& 0 \quad \text{on} \; \Gamma.
\eea
With these boundary conditions we observe that the operator $L$ has 
a kernel of dimension one, since it vanishes on constant functions. 
In order to get a uniquely solvable problem, we study the stabilized bilinear form 
$a_{\mathcal{N}} : H^1(\Omega) \times H^1(\Omega)\ra\R$ given by
\bee
a_{\mathcal{N}}(u,v) := \skp{\mathbf{C}\nabla u, \nabla v} + \skp{u,1}\skp{v,1}.
\ee

One way to formulate the finite element method for the Neumann Problem
is to use the discrete Galerkin formulation of finding $\phi_h$ such that
\been\label{eq:modelNeumann} 
a_{\mathcal{N}}(\phi_h,\psi_h) = \skp{f,\psi_h}, \quad \forall \psi_h \in S^{p,1}(\T_h)
\een
for right-hand sides $f \in L^2(\Omega)$ satisfying the solvability condition $\skp{f,1} = 0$.
Using $v\equiv 1$ as a test function the solvability condition leads to $\skp{\phi_h,1} = 0$, 
so using this formulation we derive the unique solution with integral mean zero. \\

With a basis  ${\mathcal B}_h:= \{\psi_j\, :\, j = 1,\dots, N\}$ of $S^{p,1}(\T_h)$, 
we get the symmetric, positive definite stiffness matrix $\mathbf{A}^{\mathcal{N}} \in \R^{N\times N}$ defined by
\bee
\mathbf{A}^{\mathcal{N}}_{jk} = \skp{\mathbf{C}\nabla \psi_j, \nabla \psi_k} + \skp{\psi_j,1}\skp{\psi_k,1}, 
\quad \psi_j,\psi_k \in {\mathcal B}_h,
\ee
One should note that the number $N$ of degrees of freedom is different from the number of degrees
of freedom in the mixed problem \eqref{eq:model}. In order to shorten notation, we will denote both by $N$.

With this stabilization, we have the coercivity
\been\label{eq:coercivityNeumann}
\norm{u}_{H^1(\Omega)}^2\leq C a_{\mathcal{N}}(u,u)
\een
of the bilinear form $a(\cdot,\cdot)$. \\

For an admissible block $(\tau,\sigma)$ and corresponding bounding boxes $B_{R_{\tau}},B_{R_{\sigma}}$ and 
$f \in L^2(\Omega)$ with $\supp f \subset B_{R_{\sigma}}$ we have the orthogonality relation
\been\label{eq:orthoNeumann}
a_{\mathcal{N}}(u,\psi_h) = 0, \quad \forall \psi_h \in S^{p,1}(\T_h) \, \text{with} \;\supp \psi_h \! \subset \! B_{R_{\tau}}.
\een

Since our Galerkin solution has mean zero, we can drop the stabilization term and get 
$\skp{\mathbf{C}\nabla u,\nabla \psi_h}_{L^2(B_{R_{\tau}})} = 0$.
This orthogonality and the zero mean property are captured in the following space 
\begin{eqnarray*}
\mathcal{H}^{\mathcal{N}}_h(D,\omega) &:=& \{u\in H^1(D\cap\omega) \colon \exists \widetilde{u} \in S^{p,1}(\mathcal{T}_h) \;
\mbox{s.t.} \ 
u|_{D\cap\omega} = \widetilde{u}|_{D\cap\omega}, \, \supp \widetilde{u} \subset \ol{\omega}, \\
&& \phantom{u \in }  a_{\mathcal{N}}(u,\psi_h) = 0, 
\; \forall \; \psi_h \in S^{p,1}(\mathcal{T}_h) \,\text{with} \, \supp \psi_h \subset \ol{D}\cap\ol{\omega}\} \\ 
& &  \cap \; \{u\in H^1(\Omega):\skp{u,1}_{L^2(\Omega)} = 0\}.
\end{eqnarray*}
For functions $u\in\mathcal{H}_h^{\mathcal{N}}(B_{(1+2\delta)R},\omega)$
the interior regularity result of Lemma~\ref{lem:Caccioppoli} holds as well, since 
using the orthogonality \eqref{eq:orthoNeumann} and the zero mean condition
 lead to no additional terms in comparison to the orthogonality \eqref{eq:ortho}.
Therefore, we can proceed just as in the previous chapter and derive a low rank approximation
of the Galerkin solution, which is stated in the following proposition.

\begin{proposition}\label{thm:function-approximationNeumann}
Let $(\tau,\sigma)$ be a cluster pair with bounding boxes $B_{R_\tau}$, $B_{R_\sigma}$. 
Assume $ \eta \dist(B_{R_\tau},B_{R_\sigma}) > \diam(B_{R_\tau})$ for some $\eta > 0$. 
Fix $q \in (0,1)$. 
Let $\Pi^{L^2}: L^2(\Omega) \ra S^{p,1}_0(\T_h,\Gamma_D)$ be the $L^2(\Omega)$-orthogonal projection.
Then, for each $k\in\mathbb{N}$ there exists a space
$V_k\subset S^{p,1}({\mathcal T}_h)$ with $\dim V_k\leq C_{\rm dim} (2+\eta)^d q^{-d}k^{d+1}$ 
such that for arbitrary $f \in L^{2}(\Omega)$ with 
$\supp  f  \subset B_{R_\sigma}\cap\Omega$, the solution $\phi_h$ of \eqref{eq:model} 
satisfies 
\begin{equation}
\label{eq:thm:function-approximation-1Neumann}
\min_{v \in V_k} \|\phi_h - v\|_{L^2(B_{R_{\tau}}\cap\Omega)} 
\leq C_{\rm box} q^k \|\Pi^{L^2}f\|_{L^2(\Omega)}
\leq C_{\rm box} q^k \|f\|_{L^2(B_{R_\sigma}\cap\Omega)}.
\end{equation}
The constant $C_{\rm box}>0$ depends only on $\mathbf{C}$ and $\Omega$, while
$C_{\rm dim}>0$ additionally depends on p, d, and the $\gamma$-shape regularity of $\mathcal{T}_h$.
\end{proposition}

\bpro
Since the same Caccioppoli type estimate holds, we get the same approximation result as in Lemma~\ref{lem:lowdimapp},
and we can proceed as in the proof of Proposition~\ref{thm:function-approximation}.
\epro

This approximation result can be transferred to the matrix level exactly
 in the same way as in Section~\ref{sec:H-matrix-approximation}, where the mixed boundary value problem
\eqref{eq:modelstrong} is discussed,
 to derive an $\H$-matrix approximant for the matrix $(\mathbf{A}^{\mathcal{N}})^{-1}$.

\section{Proof of main results}
\label{sec:H-matrix-approximation}
We use the approximation of $\phi_h$ from the low dimensional spaces 
given in Proposition~\ref{thm:function-approximation} 
to construct a blockwise low-rank approximation of $\mathbf{A}^{-1}$ 
and in turn an $\mathcal{H}$-matrix approximation of $\mathbf{A}^{-1}$. 
In fact, we will only use a FEM-isomorphism to transfer Proposition~\ref{thm:function-approximation} 
to the matrix level, which follows the lines of \cite[Theorem 2]{Boerm}. \\

\bpro[of Theorem~\ref{th:blockapprox}]
If $C_{\rm dim} (2+\eta)^d q^{-d} k^{d+1} \geq \min (\abs{\tau},\abs{\sigma})$, we use the exact matrix block 
$\mathbf{X}_{\tau\sigma}=\mathbf{A}^{-1}|_{\tau \times \sigma}$ and $\mathbf{Y}_{\tau\sigma} = \mathbf{I} \in \mathbb{R}^{\abs{\sigma}\times\abs{\sigma}}$. 

If $C_{\rm dim} (2+\eta)^d q^{-d} k^{d+1}  < \min (\abs{\tau},\abs{\sigma})$,
let  $\lambda_i:L^2(\Omega) \rightarrow \mathbb{R}$ be continuous linear functionals on $L^2(\Omega)$
 satisfying $\lambda_i(\psi_j) = \delta_{ij}$.
 We define $\R^{\tau}:= \{\mathbf{x} \in \R^N \; : \; x_i = 0 \; \forall \; i \notin \tau \}$ and the mappings 
 \begin{equation*}
\Lambda_{\tau} : L^2(\Omega) \rightarrow \mathbb{R}^{\tau}, v \mapsto (\lambda_i(v))_{i \in \tau} \;\text{and} 
\; \mathcal{J}_{\tau}: \mathbb{R}^{\tau} \rightarrow S^{p,1}_0({\mathcal T}_h,\Gamma_D),\; \mathbf{x} \mapsto \sum_{j\in\tau} x_j \psi_j.
\end{equation*}
For $\mathbf{x} \in \R^{\tau}$, \eqref{eq:basis} leads to the stability estimate 
\been\label{eq:isomorphism}
h^{d/2}\norm{\mathbf{x}}_2 \lesssim \norm{\mathcal{J}_{\tau}\mathbf{x}}_{L^2(\Omega)} \lesssim h^{d/2}\norm{\mathbf{x}}_2.
\een
Let $V_k$ be the finite dimensional subspace from Proposition~\ref{thm:function-approximation}. 

Because of \eqref{eq:isomorphism} and the $L^2$-stability of $\mathcal{J}_{\mathcal{I}} \Lambda_{\mathcal{I}}$,
the adjoint $\Lambda_{\mathcal{I}}^* : \R^N \ra L^2(\Omega)'$ of 
$\Lambda_{\mathcal{I}}$ satisfies
\bee
\norm{\Lambda_{\mathcal{I}}^* \mathbf{b}}_{L^2(\Omega)} = \sup_{v \in L^2(\Omega)}\frac{\skp{\mathbf{b},\Lambda_{\mathcal{I}} v}_2}{\norm{v}_{L^2(\Omega)}} 
  \lesssim \norm{\mathbf{b}}_2 \sup_{v \in L^2(\Omega)}\frac{h^{-d/2}\norm{\mathcal{J}_{\mathcal{I}}\Lambda_{\mathcal{I}} v}_{L^2(\Omega)}}{\norm{v}_{L^2(\Omega)}} \leq C h^{-d/2}\norm{\mathbf{b}}_2. 
\ee
Moreover, if $\mathbf{b} = (\skp{f,\psi_i})_{i\in {\mathcal{I}}}$, we have 
$(\Lambda_{\mathcal{I}}^* \mathbf{b})(\psi_i) = b_i = \skp{f,\psi_i} = \skp{\Pi^{L^2}f,\psi_i}$. 
Therefore, $f$ and $\Lambda_{\mathcal{I}}^* \mathbf{b} = \Pi^{L^2} f$ have the same 
Galerkin approximation.

Let $V_k$ be the finite dimensional subspace from Proposition~\ref{thm:function-approximation}. 
We define $\mathbf{X}_{\tau\sigma}$ as an orthogonal basis of the space 
$\mathcal{V}_{\tau} := \{\Lambda_{\tau} v \; : \; v \in V_k \}$ and 
$\mathbf{Y}_{\tau\sigma} := \mathbf{A}^{-1}|_{\tau\times \sigma}^{T}\mathbf{X}_{\tau\sigma}$.
 Then, the rank of $\mathbf{X}_{\tau\sigma},\mathbf{Y}_{\tau\sigma}$ 
is bounded by $\dim V_k \leq C_{\rm dim}(2+\eta)^d q^{-d} k^{d+1}$. 

The estimate \eqref{eq:isomorphism} and 
the approximation result from Proposition~\ref{thm:function-approximation} 
provide the error estimate 
\bea
\norm{\Lambda_{\tau} \phi_h - \Lambda_{\tau} v}_2 &\lesssim& h^{-d/2}\norm{\mathcal{J}_{\tau}(\Lambda_{\tau} \phi_h-\Lambda_{\tau} v)}_{L^2(\Omega)} 
\leq h^{-d/2}\norm{\phi_h-v}_{L^2(B_{R_{\tau}}\cap\Omega)} \\
  &\leq& C_{\rm box} h^{-d/2}q^k\norm{\Pi^{L^2} f}_{L^2(\Omega)} \lesssim C_{\rm box} h^{-d}q^k\norm{\mathbf{b}}_{2}.
\eea
Since $\mathbf{X}_{\tau\sigma} \mathbf{X}_{\tau\sigma}^T$ is the orthogonal projection from $\R^N$ onto $\mathcal{V}_{\tau}$, 
we get that $z:=\mathbf{X}_{\tau\sigma} \mathbf{X}_{\tau\sigma}^T \Lambda_{\tau} \phi_h$ 
is the best approximation of $\Lambda_{\tau} \phi_h$ in $\mathcal{V}_{\tau}$ and arrive at 
\bee
\norm{\Lambda_{\tau} \phi_h-z}_2 \leq \norm{\Lambda_{\tau} \phi_h-\Lambda_{\tau} v}_2 \lesssim C_{\rm box} N q^k\norm{\mathbf{b}}_2.
\ee
If we define $\mathbf{Y}_{\tau,\sigma} := \mathbf{A}^{-1}|_{\tau\times \sigma}^{T}\mathbf{X}_{\tau\sigma}$, 
we get $z = \mathbf{X}_{\tau\sigma} \mathbf{Y}_{\tau\sigma}^T \mathbf{b}$, 
since $\Lambda_{\tau} \phi_h = \mathbf{A}^{-1}|_{\tau\times \sigma}\mathbf{b}$.
\epro

The following lemma gives an estimate for the global spectral norm by the local spectral norms, which we will use in combination with Theorem~\ref{th:blockapprox} to derive our main result, Theorem~\ref{th:Happrox}. 

\blem[{\cite[Lemma 6.5.8]{GrasedyckDissertation,Hackbusch}}]\label{lem:spectralnorm}
Let $\mathbf{M} \in \R^{N\times N}$ and $P$ be a partitioning of ${\mathcal{I}}\times {\mathcal{I}}$. Then,
\bee
\norm{\mathbf{M}}_2 \leq C_{\rm sp} \left(\sum_{\ell=0}^{\infty}\max\{\norm{\mathbf{M}|_{\tau\times \sigma}}_2 : (\tau,\sigma) \in P, \level(\tau) = \ell\}\right).
\ee
\elem

Now we are able to prove our main result, Theorem~\ref{th:Happrox}.


\bpro[of Theorem \ref{th:Happrox}]
For each admissible cluster pair $(\tau,\sigma)$, 
Theorem \ref{th:blockapprox} provides matrices $\mathbf{X}_{\tau\sigma} \in \R^{\abs{\tau}\times r}$, $\mathbf{Y}_{\tau\sigma} \in \R^{r \times\abs{\sigma}}$, 
so that we can define the $\H$-matrix $\mathbf{V}_{\H}$ by 
\bee
\mathbf{B}_{\H} = \left\{
\begin{array}{l}
 \mathbf{X}_{\tau\sigma}\mathbf{Y}_{\tau\sigma}^T \quad \;\textrm{if}\hspace{2mm} (\tau,\sigma) \in P_{\text{far}}, \\
 \mathbf{A}^{-1}|_{\tau \times \sigma} \quad \textrm{otherwise}.
 \end{array}
 \right.
\ee
On each admissible block $(\tau,\sigma) \in \Pfar$, we can use the blockwise estimate of Theorem~\ref{th:blockapprox} and get for $q\in(0,1)$
\bee
\norm{(\mathbf{A}^{-1} - \mathbf{B}_{\H})|_{\tau \times \sigma}}_2 \leq C_{\rm apx}N q^k.
\ee
On inadmissible blocks, the error is zero by definition. Therefore, Lemma~\ref{lem:spectralnorm} concludes the proof, since
\bea
\norm{\mathbf{A}^{-1} - \mathbf{B}_{\H}}_2 &\leq& C_{\rm sp} \left(\sum_{\ell=0}^{\infty}\text{max}\{\norm{(\mathbf{A}^{-1} - \mathbf{B}_{\H})|_{\tau \times \sigma}}_2 : (\tau,\sigma) \in P, \level(\tau) = \ell\}\right) \\
 &\leq& C_{\rm apx} C_{\rm sp} N q^k \depth(\mathbb{T}_{\mathcal{I}}).
\eea
Defining $b = -\frac{\ln(q)}{C_{\rm dim}^{1/(d+1)}}q^{d/(d+1)}(2+\eta)^{-d/(1+d)} > 0$, we obtain $q^k = e^{-br^{1/(d+1)}}$ and hence
\begin{equation*}
\norm{\mathbf{A}^{-1} - \mathbf{B}_{\mathcal{H}}}_2 \leq C_{\rm apx}C_{\rm sp} N {\rm depth}(\mathbb{T}_{\mathcal{I}})e^{-br^{1/(d+1)}},
\end{equation*}
which concludes the proof.
\epro

\section{Hierarchical $LU$-decomposition}
\label{sec:hierarchical-LU-decomposition}
In \cite{Bebendorf07} the existence of an (approximate) $\H$-$LU$ decomposition,
i.e., a factorization of the form $\mathbf{A} \approx \mathbf{L}_{\H}\mathbf{U}_{\H}$ 
with lower and upper triangular $\H$-matrices
$\mathbf{L}_{\H}$ and $\mathbf{U}_{\H}$, was proven for finite element matrices $\mathbf{A}$ 
corresponding to the Dirichlet problem
for elliptic operators with $L^{\infty}$-coefficients. In \cite{GrasedyckKriemannLeBorne} this result
was extended to the case, where the block structure of the $\H$-matrix is constructed by domain decomposition 
clustering methods, instead of the standard geometric bisection clustering.

Algorithms for computing an $\H$-$LU$ decomposition have been proposed repeatedly in the literature, e.g., 
\cite{Lintner,Bebendorf05} and numerical evidence for their usefulness put forward; we mention here 
that $\H$-$LU$ decomposition can be employed for black box preconditioning in iterative solvers, 
\cite{Bebendorf05,Grasedyck05,Grasedyck08,leborne-grasedyck06,grasedyck-kriemann-leborne08}.
An existence result for $\H$-$LU$ factorization is then an important step towards a mathematical understanding 
of the good performance of these schemes. 

The main steps in the proof of \cite{Bebendorf07} are to approximate
certain Schur complements of $\mathbf{A}$ by $\H$-matrices
and to show a recursion formula for the Schur complement. 
Using these two observations an approximation of the exact $LU$-factors for the Schur complements, and 
consequently for the whole matrix, can be derived recursively.

Since the construction of the approximate $LU$-factors is completely algebraic, once we know that 
the Schur complements have an $\H$-matrix approximation of arbitrary accuracy, we will show
that we can provide such an approximation and only sketch the remaining steps.
Details can be found in \cite{Bebendorf07,GrasedyckKriemannLeBorne}.

Our main result, Theorem~\ref{th:Happrox}, shows the existence of an $\H$-matrix approximation to the inverse
FEM stiffness matrix with arbitrary accuracy, whereas previous results achieve accuracy up to the finite element error.
In fact, both \cite{Bebendorf07,GrasedyckKriemannLeBorne} assume, in order to derive an $\H$-$LU$ decomposition, that 
approximations to the inverse with arbitrary accuracy exist. 
Thus, due to our main result this assumption is fulfilled for inverse finite element matrices for elliptic
operators with various boundary conditions.   \newline

Since we are in the setting of the Lax-Milgram Lemma, we get that the, in general, non symmetric matrix
$\mathbf{A}$ is positive definite in the sense that $\mathbf{x}^T\mathbf{A}\mathbf{x}>0$ for all $\mathbf{x} \neq 0$.
Therefore, $\mathbf{A}$ has an $LU$-decomposition $\mathbf{A}=\mathbf{L}\mathbf{U}$, 
where $\mathbf{L}$ is a lower triangular matrix and $\mathbf{U}$ is 
an upper triangular matrix, independently of the numbering of the degrees of freedom, i.e., every other 
numbering of the basis functions permits an $LU$-decomposition as well (see, e.g., \cite[Cor.~{3.5.6}]{horn-johnson13}). 
By classical linear algebra (see, e.g., \cite[Cor.~{3.5.6}]{horn-johnson13}), 
this implies that for any $n \leq N$ and index set 
$\rho := \{1,\ldots,n\}$, the matrix $\mathbf{A}|_{\rho \times \rho}$ is invertible. \newline


We start with the approximation of appropriate Schur complements. 

\subsection{Schur complements}

One way to approximate the Schur complement for a finite element matrix
is to follow the lines of \cite{Bebendorf07,GrasedyckKriemannLeBorne}
by using $\H$-arithmetics and the sparsity of the finite element matrix.
We present a different way of deriving such a result, which is more in line with our
procedure in Section~\ref{sec:Approximation-solution}. It relies on interpreting Schur complements 
as FEM stiffness matrices from constrained spaces. 

\blem\label{lem:Schur}
Let $(\tau,\sigma)$ be an admissible cluster pair and $\rho := \{i\in \mathcal{I} : i < \min(\tau\cup\sigma)\}$. 
Define the Schur complement 
$\mathbf{S}(\tau,\sigma) = \mathbf{A}|_{\tau\times\sigma} - \mathbf{A}|_{\tau\times \rho} (\mathbf{A}|_{\rho\times \rho})^{-1}\mathbf{A}|_{\rho\times\sigma}$.
Then, there exists a rank-$r$ matrix $\mathbf{S}_{\H}(\tau,\sigma)$ such that 
\bee
\norm{\mathbf{S}(\tau,\sigma) - \mathbf{S}_{\H}(\tau,\sigma)}_2 \leq C_{\rm sc} h^{-1} e^{-br^{1/(d+1)}}\norm{\mathbf{A}}_2,
\ee
where the constant $C_{\rm sc}>0$ depends only on the boundary value problem \eqref{eq:modelstrong}, $\Omega$,
 $p$, $d$, and the $\gamma$-shape regularity of $\mathcal{T}_h$.
\elem

\bpro
We define $\omega_{\rho} = {\rm interior}\left( \bigcup_{i\in \rho}\supp \psi_i \right) \subset \Omega$ 
and let $B_{R_{\tau}},B_{R_{\sigma}}$ be
bounding boxes for the clusters $\tau$, $\sigma$ with \eqref{eq:admissibility}. 
Our starting point is the observation that the Schur complement matrix
$\mathbf{S}(\tau,\sigma)$ can be understood in terms of an orthogonalization
with respect to the degrees of freedom in $\rho$. That is, 
for $\mathbf{u}\in\R^{\abs{\tau}},\mathbf{w}\in\R^{\abs{\sigma}}$  a direct calculation shows
\been\label{eq:SchurRepresentation}
\mathbf{u}^T \mathbf{S}(\tau,\sigma)\mathbf{w} = a(\widetilde{u},w), 
\een
with $w = \sum_{j=1}^{\abs{\sigma}}\mathbf{w}_j\psi_{j_{\sigma}}$, 
where the index $j_{\sigma}$ denotes the $j$-th basis function corresponding to the cluster $\sigma$,
and the function $\widetilde{u} \in S^{p,1}_0(\T_h,\Gamma_D)$ 
is defined by $\widetilde{u} = \sum_{j=1}^{\abs{\tau}}\mathbf{u}_j\psi_{j_{\tau}} + u_{\rho}$ 
with $\supp u_{\rho} \subset \overline{\omega_{\rho}}$ such that 
\been\label{eq:SchurOrthogonality}
a(\widetilde{u},w) = 0 \quad \forall w \in S^{p,1}_0({\mathcal T}_h,\Gamma_D) \; \text{with}\; 
\supp w \subset \overline{\omega_{\rho}}.
\een
The key to approximate the Schur complement $\mathbf{S}(\tau,\sigma)$ is
to approximate the function $\widetilde{u}$.
We will provide such an approximation by applying the techniques from the previous chapters with the use of the 
orthogonality \eqref{eq:SchurOrthogonality}.

Since $\supp \widetilde{u} \subset B_{R_{\tau}} \cup \overline{\omega_{\rho}}$, 
we get for $w$ with $\supp w \subset B_{R_{\sigma}}$ that
\bee
a(\widetilde{u},w) = a(\widetilde{u}|_{\supp w},w) = a(\widetilde{u}|_{B_{R_{\sigma}}\cap\omega_{\rho}},w).
\ee
Therefore, we only need to approximate $\widetilde{u}$ on the intersection $B_{R_{\sigma}}\cap \omega_{\rho}$.
This support property and the orthogonality \eqref{eq:SchurOrthogonality} imply that 
$\widetilde{u} \in \H_h(B_{(1+\delta)R_{\sigma}},\omega_{\rho})$. 

Therefore, Lemma~\ref{lem:Caccioppoli} can be applied to $\widetilde{u}$. As a consequence,
Lemma~\ref{cor:lowdimapp} provides a low dimensional space $ V_k$, where the choice $\kappa =\frac{1}{\eta+1}$
bounds the dimension of $V_k$ by $\dim V_k \leq C_{\rm dim}(2+\eta)^dq^{-d}k^{d+1}$.
Moreover, the best approximation $\widetilde{v} = \Pi_{V_k}\widetilde{u} \in V_k$
to $\widetilde{u}$ in the space $V_k$ satisfies
\bee
\triplenorm{\widetilde{u}-\widetilde{v}}_{h,(1+\delta)R_{\sigma}} \leq q^k \triplenorm{\widetilde{u}}_{h,(1+\delta)R_{\sigma}}.
\ee
This implies 
\bea
\abs{a(\widetilde{u},w)-a(\widetilde{v},w)}&\lesssim& 
\norm{\widetilde{u}-\widetilde{v}}_{H^1(B_{(1+\delta)R_{\sigma}}\cap\omega_{\rho})}\norm{w}_{H^1(B_{(1+\delta)R_{\sigma}}\cap\Omega)}\\
&\lesssim& \frac{R_{\sigma}}{h} \triplenorm{\widetilde{u}-\widetilde{v}}_{h,(1+\delta)R_{\sigma}}\norm{w}_{H^1(\Omega)} 
\lesssim h^{-1}q^k\norm{\widetilde{u}}_{H^1(\Omega)}\norm{w}_{H^1(\Omega)}.
\eea
Since $\supp(\widetilde{u}-u) = \supp(u_{\rho}) \subset \overline{\omega_{\rho}}$ with $u = \sum_{j=1}^{\abs{\tau}}\mathbf{u}_j\psi_{j_{\tau}}$, 
the coercivity \eqref{eq:coercivity} and orthogonality \eqref{eq:SchurOrthogonality} lead to
\bee
\norm{\widetilde{u}-u}_{H^1(\Omega)}^2\lesssim a(\widetilde{u}-u,\widetilde{u}-u) 
= a(-u,\widetilde{u}-u) \lesssim \norm{u}_{H^1(\Omega)}\norm{\widetilde{u}-u}_{H^1(\Omega)}.
\ee
Consequently, we get with an inverse estimate and \eqref{eq:isomorphism} that
\bea
\abs{a(\widetilde{u},w)-a(\widetilde{v},w)} &\lesssim& h^{-1}q^k\left(\norm{\widetilde{u}-u}_{H^1(\Omega)}+\norm{u}_{H^1(\Omega)}\right)\norm{w}_{H^1(\Omega)} \\
&\lesssim& h^{-1}q^k\norm{u}_{H^1(\Omega)}\norm{w}_{H^1(\Omega)} \lesssim h^{d-3}q^k\norm{\mathbf{u}}_2\norm{\mathbf{w}}_2.
\eea
The linear mapping $\mathcal{E}: u\mapsto\widetilde{v}$ with $\dim \text{ran}\; \mathcal{E} \leq \Cdim (2+\eta)^d q^{-d}k^{d+1}$ 
has a matrix representation $\mathbf{u}\mapsto \mathbf{B}\mathbf{u}$, where the rank of $\mathbf{B}$ is bounded by 
$\Cdim (2+\eta)^d q^{-d}k^{d+1}$. 
Therefore, we get that $a(\mathcal{E}u,w) = \mathbf{u}^T\mathbf{B}^T\mathbf{A}|_{\tau\times\sigma}\mathbf{w}$.
The definition $\mathbf{S_{\H}}(\tau,\sigma) := \mathbf{B}^T\mathbf{A}|_{\tau\times\sigma}$ leads to a 
matrix $\mathbf{S_{\H}}(\tau,\sigma)$ of rank $r \leq \Cdim (2+\eta)^d q^{-d}k^{d+1}$ such that
\bee
\norm{\mathbf{S}(\tau,\sigma)-\mathbf{S}_{\mathcal{H}}(\tau,\sigma)}_2 = 
\sup_{\mathbf{u}\in\R^{\abs{\tau}},\mathbf{w}\in \R^{\abs{\sigma}}} 
\frac{\abs{\mathbf{u}^T(\mathbf{S}(\tau,\sigma)-\mathbf{S}_{\mathcal{H}}(\tau,\sigma))\mathbf{w}}}{\norm{\mathbf{u}}_2\norm{\mathbf{w}}_2} 
\leq C h^{d-3}e^{-br^{1/(d+1)}},
\ee
and the estimate $\frac{1}{\norm{\mathbf{A}}_2}\lesssim h^{2-d}$ from \cite[Theorem 2]{ern-guermond} finishes the proof.
\epro

We refer to the next subsection for the existence of the inverse $\mathbf{S}(\tau,\tau)^{-1}$ of the Schur complement 
$\mathbf{S}(\tau,\tau)$. We proceed to approximate it by blockwise rank-$r$ matrices.  
With the representation of the Schur complement from \eqref{eq:SchurRepresentation}, we get that for a given
right-hand side $f \in L^2(\Omega)$,
solving $\mathbf{S}(\tau,\tau)\mathbf{u} = \mathbf{f}$
with $\mathbf{f} \in \R^{\abs{\tau}}$ defined by $\mathbf f_i = \skp{f,\psi_{i_{\tau}}}$,
is equivalent to solving $a(\widetilde{u},w) = \skp{f,w}$ 
for all $w \in S^{p,1}_0(\T_h,\Gamma_D)$ with $\supp w \subset \overline{\omega_{\tau}}$. 
Let $\tau_1\times\sigma_1 \subset \tau\times\tau$ be an $\eta$-admissible subblock.
For $f \in L^2(\Omega)$ with $\supp f \subset B_{R_{\sigma_1}}$, we get the orthogonality
\bee
a(\widetilde{u},w) = 0 \quad \forall w \in S^{p,1}_0(\T_h,\Gamma_D), 
\supp w \subset B_{R_{\tau_1}}\cap\overline{\omega_{\tau}}.
\ee
Therefore, we have $\widetilde{u} \in \H_h(B_{R_{\tau_1}},\omega_{\tau})$ and
 our results from Section \ref{sec:Approximation-solution} can be applied to approximate $\widetilde{u}$
on $B_{R_{\tau_1}}\cap\omega_{\tau}$. As in Section \ref{sec:H-matrix-approximation}, 
this approximation can be used to construct a rank-$r$ factorization of the subblock 
$\mathbf{S}(\tau,\tau)^{-1}|_{\tau_1\times\sigma_1}$, 
which is stated in the following theorem.

\begin{theorem}
Let $\tau \subset \mathcal{I}$ and $\rho := \{i\in \mathcal{I} : i < \min(\tau)\}$ and
$\tau_1\times\sigma_1 \subset \tau\times\tau$ be $\eta$-admissible.
Define the Schur complement 
$\mathbf{S}(\tau,\tau) = \mathbf{A}|_{\tau\times\tau} - \mathbf{A}|_{\tau\times \rho} (\mathbf{A}|_{\rho\times \rho})^{-1}\mathbf{A}|_{\rho\times\tau}$.
Then, there exist rank-$r$ matrices 
$\mathbf{X}_{\tau_1\sigma_1} \in \R^{\abs{\tau_1}\times r}$, $\mathbf{Y}_{\tau_1\sigma_1} \in \R^{\abs{\sigma_1}\times r}$
such that 
\begin{equation}
\norm{\mathbf{S}(\tau,\tau)^{-1}|_{\tau_1\times\sigma_1} - \mathbf{X}_{\tau_1\sigma_1}\mathbf{Y}_{\tau_1\sigma_1}^T}_2 
\leq C_{\rm apx} N e^{-br^{1/(d+1)}}.
\end{equation}
The constants $C_{\rm apx},b>0$ depend only on the boundary value problem \eqref{eq:modelstrong}, 
$\Omega$, $d$, $p$, and the $\gamma$-shape regularity of $\mathcal{T}_h$.
\end{theorem}

\subsection{Existence of $\H$-$LU$ decomposition}
\label{sec:LU-decomposition}
In this subsection, we will use the approximation of the Schur complement from the previous section
to prove the existence of an (approximate) $\H$-$LU$ decomposition. 
We start with a hierarchical relation of the Schur complements $\mathbf{S}(\tau,\tau)$. \newline 

The Schur complements $\mathbf{S}(\tau,\tau)$ for a block $\tau \in \mathbb{T}_{\mathcal{I}}$ 
can be derived from the Schur complements of its sons by 
\bee
\mathbf{S}(\tau,\tau) = \begin{pmatrix} \mathbf{S}(\tau_1,\tau_1) & \mathbf{S}(\tau_1,\tau_2) \\ 
\mathbf{S}(\tau_2,\tau_1) & \mathbf{S}(\tau_2,\tau_2) + \mathbf{S}(\tau_2,\tau_1)\mathbf{S}(\tau_1,\tau_1)^{-1}\mathbf{S}(\tau_1,\tau_2) \end{pmatrix},
\ee
where $\tau_1,\tau_2$ are the sons of $\tau$.
A proof of this relation can be found in \cite[Lemma 3.1]{Bebendorf07}. One should note that
the proof does not use any properties of the matrix $\mathbf{A}$ other than invertibility and 
existence of an $LU$-decomposition. 
Moreover, we have by definition of $\mathbf{S}(\tau,\tau)$ that $\mathbf{S}(\mathcal{I},\mathcal{I}) = \mathbf{A}$.

If $\tau$ is a leaf, we get the $LU$-decomposition of $\mathbf{S}(\tau,\tau)$ by the classical 
$LU$-decomposition,
which exists since $\mathbf{A}$ has an $LU$-decomposition.
If $\tau$ is not a leaf, we use the hierarchical relation of the Schur complements to
define an $LU$-decomposition of the Schur complement
 $\mathbf{S}(\tau,\tau)$ by 
\been\label{eq:LUdefinition}
\mathbf{L}(\tau) := \begin{pmatrix} \mathbf{L}(\tau_1) & 0 \\ \mathbf{S}(\tau_2,\tau_1)\mathbf{U}(\tau_1)^{-1} & \mathbf{L}(\tau_2)  \end{pmatrix}, \quad 
\mathbf{U}(\tau) := \begin{pmatrix} \mathbf{U}(\tau_1) & \mathbf{L}(\tau_1)^{-1}\mathbf{S}(\tau_1,\tau_2) \\ 0 & \mathbf{U}(\tau_2) \end{pmatrix},
\een
with $\mathbf{S}(\tau_1,\tau_1) = \mathbf{L}(\tau_1)\mathbf{U}(\tau_1)$,
 $\mathbf{S}(\tau_2,\tau_2) = \mathbf{L}(\tau_2)\mathbf{U}(\tau_2)$ and indeed get 
$\mathbf{S}(\tau,\tau) = \mathbf{L}(\tau) \mathbf{U}(\tau)$.
Moreover, the uniqueness of the $LU$-decomposition of $\mathbf{A}$ implies that due to
$\mathbf{L}\mathbf{U} = \mathbf{A} = \mathbf{S}(\mathcal{I},\mathcal{I}) = \mathbf{L}(\mathcal{I})\mathbf{U}(\mathcal{I})$, we have
$\mathbf{L} = \mathbf{L}(\mathcal{I})$ and $\mathbf{U} = \mathbf{U}(\mathcal{I})$.

The existence of the inverses $\mathbf{L}(\tau_1)^{-1}$ and $\mathbf{U}(\tau_1)^{-1}$ 
follows by induction over the levels, since on a leaf the existence is clear and the matrices 
$\mathbf{L}(\tau)$, $\mathbf{U}(\tau)$ are block triangular matrices. Consequently, the inverse of
$\mathbf{S}(\tau,\tau)$ exists. 

Moreover, the restriction of the lower triangular part $\mathbf{S}(\tau_2,\tau_1)\mathbf{U}(\tau_1)^{-1}$ 
of the matrix $\mathbf{L}(\tau)$ to a subblock $\tau_2'\times\tau_1'$
with $\tau_i'$ a son of $\tau_i$ satisfies
\bee
\left(\mathbf{S}(\tau_2,\tau_1)\mathbf{U}(\tau_1)^{-1}\right)|_{\tau_2'\times\tau_1'} = 
\mathbf{S}(\tau_2',\tau_1')\mathbf{U}(\tau_1')^{-1},
\ee
and the upper triangular part of $\mathbf{U}(\tau)$ satisfies a similar relation. \newline

The following Lemma shows that the spectral norm of the inverses 
$\mathbf{L}(\tau)^{-1}$, $\mathbf{U}(\tau)^{-1}$ can be bounded by the norm of the inverses
$\mathbf{L}(\mathcal{I})^{-1}$, $\mathbf{U}(\mathcal{I})^{-1}$.

\begin{lemma}\label{lem:LUnorm}
For $\tau\in \mathbb{T}_{\mathcal{I}}$,
let $\mathbf{L}(\tau)$, $\mathbf{U}(\tau)$ be given by \eqref{eq:LUdefinition}. Then,
\bea
\max_{\tau\in\mathbb{T}_{\mathcal{I}}}\norm{\mathbf{L}(\tau)^{-1}}_2 &=& \norm{\mathbf{L}(\mathcal{I})^{-1}}_2, \\
\max_{\tau\in\mathbb{T}_{\mathcal{I}}}\norm{\mathbf{U}(\tau)^{-1}}_2 &=& \norm{\mathbf{U}(\mathcal{I})^{-1}}_2.
\eea
\end{lemma}
\begin{proof}
We only show the result for $\mathbf{L}(\tau)$. With the block structure of \eqref{eq:LUdefinition} we get the inverse
\bee
\mathbf{L}(\tau)^{-1} = \begin{pmatrix} \mathbf{L}(\tau_1)^{-1} & 0 \\ -\mathbf{L}(\tau_2)^{-1} \mathbf{S}(\tau_2,\tau_1)\mathbf{U}(\tau_1)^{-1}\mathbf{L}(\tau_1)^{-1} & \mathbf{L}(\tau_2)^{-1} \end{pmatrix}.
\ee
So, we get by choosing $\mathbf{x}$ such that $\mathbf{x}_i = 0$ for $i \in \tau_1$ that
\bea
\norm{\mathbf{L}(\tau)^{-1}}_2 = \sup_{\mathbf{x}\in \R^{\abs{\tau}},\norm{x}_2=1}\norm{\mathbf{L}(\tau)^{-1}\mathbf{x}}_2 
\geq \sup_{\mathbf{x}\in \R^{\abs{\tau_2}},\norm{x}_2=1}\norm{\mathbf{L}(\tau_2)^{-1}\mathbf{x}}_2 
= \norm{\mathbf{L}(\tau_2)^{-1}}_2.
\eea
The same argument for $\left(\mathbf{L}(\tau)^{-1}\right)^T$ leads to 
\bea
\norm{\mathbf{L}(\tau)^{-1}}_2 = \norm{\left(\mathbf{L}(\tau)^{-1}\right)^T}_2 \geq \norm{\mathbf{L}(\tau_1)^{-1}}_2.
\eea
Thus, we have $\norm{\mathbf{L}(\tau)^{-1}}_2 \geq \max_{i=1,2}\norm{\mathbf{L}(\tau_1)^{-1}}_2$ and as a consequence
$\max_{\tau\in\mathbb{T}_{\mathcal{I}}}\norm{\mathbf{L}(\tau)^{-1}}_2 = \norm{\mathbf{L}(\mathcal{I})^{-1}}_2$.
\end{proof}

We can now formulate the existence result for an $\H$-$LU$ decomposition.

\begin{theorem}\label{th:HLU}
Let $\mathbf{A} = \mathbf{L}\mathbf{U}$ with $\mathbf{L},\mathbf{U}$ being lower and upper triangular matrices.
There exist lower and upper triangular blockwise rank-$r$ matrices $\mathbf{L_{\H}},\mathbf{U_{\H}}$ such that 
\bean
\norm{\mathbf{A}-\mathbf{L_{\H}}\mathbf{U_{\H}}}_2 &\leq& \Big(C_{\rm LU} h^{-1} {\rm depth}(\mathbb{T}_{\mathcal{I}}) e^{-br^{1/(d+1)}} \\
 & & \;\; + C_{\rm LU}^2 h^{-2} {\rm depth}(\mathbb{T}_{\mathcal{I}})^2 e^{-2br^{1/(d+1)}}\Big)\norm{\mathbf{A}}_2, \nonumber
\eean
where $C_{\rm LU} = C_{\rm sp}C_{\rm apx}(\kappa_2(\mathbf{U})+\kappa_2(\mathbf{L}))$, with the constant $C_{\rm apx}$ 
from Theorem~\ref{th:blockapprox} and the spectral condition numbers $\kappa_2(\mathbf{U})$, $\kappa_2(\mathbf{L})$. 
\end{theorem}
\bpro
With Lemma~\ref{lem:Schur}, we get a low rank approximation of an admissible subblock $\tau'\times\sigma'$ of 
the lower triangular part of $\mathbf{L}(\tau)$ by
\bea
\norm{\mathbf{S}(\tau,\sigma)\mathbf{U}(\sigma)^{-1}|_{\tau'\times\sigma'}\! - \! \mathbf{S}_{\H}(\tau',\sigma')\mathbf{U}(\sigma')^{-1}}_2 \!\! &=& \!\!
\norm{\mathbf{S}(\tau',\sigma')\mathbf{U}(\sigma')^{-1} - \mathbf{S}_{\H}(\tau',\sigma')\mathbf{U}(\sigma')^{-1}}_2 \\
&\leq&\!\! C_{\rm apx} h^{-1}  e^{-br^{1/(d+1)}}\norm{\mathbf{U}(\sigma')^{-1}}_2\norm{\mathbf{A}}_2.
\eea 
Since $\mathbf{S}_{\H}(\tau',\sigma')\mathbf{U}(\sigma')^{-1}$ is a rank-$r$ matrix, Lemma~\ref{lem:spectralnorm}
immediately provides an $\H$-matrix approximation $\mathbf{L}_{\H}$
of the $LU$-factor $\mathbf{L}(\mathcal{I}) = \mathbf{L}$. Therefore, with Lemma~\ref{lem:LUnorm} 
we get
\bee
\norm{\mathbf{L}-\mathbf{L}_{\H}}_2\leq C_{\rm apx} C_{\rm sp}h^{-1}\depth(\mathbb{T}_{\mathcal{I}})e^{-br^{1/(d+1)}}
\norm{\mathbf{U}^{-1}}_2\norm{\mathbf{A}}_2
\ee
and in the same way an $\H$-matrix approximation $\mathbf{U}_{\H}$ of $\mathbf{U}(\mathcal{I}) = \mathbf{U}$ with
\bee
\norm{\mathbf{U}-\mathbf{U}_{\H}}_2\leq C_{\rm apx} C_{\rm sp}h^{-1} \depth(\mathbb{T}_{\mathcal{I}})e^{-br^{1/(d+1)}} 
\norm{\mathbf{L}^{-1}}_2\norm{\mathbf{A}}_2.
\ee
Since $\mathbf{A}=\mathbf{L}\mathbf{U}$, the triangle inequality finally leads to
\bea
\norm{\mathbf{A}-\mathbf{L}_{\H}\mathbf{U}_{\H}}_2 &\leq& \norm{\mathbf{L}-\mathbf{L}_{\H}}_2\norm{\mathbf{U}}_2 + \norm{\mathbf{U}-\mathbf{U}_{\H}}_2 \norm{\mathbf{L}}_2 + \norm{\mathbf{L}-\mathbf{L}_{\H}}_2\norm{\mathbf{U}-\mathbf{U}_{\H}}_2 \\
&\lesssim& \left(\kappa_2(\mathbf{U})+\kappa_2(\mathbf{L})\right)\depth(\mathbb{T}_{\mathcal{I}})h^{-1} e^{-br^{1/(d+1)}}\norm{\mathbf{A}}_2\\
 & & +\kappa_2(\mathbf{U}) \kappa_2(\mathbf{L})\depth(\mathbb{T}_{\mathcal{I}})^2h^{-2} e^{-2br^{1/(d+1)}}\frac{\norm{\mathbf{A}}_2^2}{\norm{\mathbf{L}}_2\norm{\mathbf{U}}_2},
\eea
and the estimate $\norm{\mathbf{A}}_2 \leq \norm{\mathbf{L}}_2\norm{\mathbf{U}}_2$ finishes the proof.
\epro

In the symmetric case, we may use the weaker admissibility condition \eqref{eq:minadmissible}
instead of \eqref{eq:admissibility} and obtain a result analogously to that of Theorem~\ref{th:HLU}
for the Cholesky decomposition.

\begin{corollary}
Let $\mathbf{b} = \mathbf{0}$ in \eqref{eq:ellipticoperatorLO}
so that the resulting Galerkin matrix $\mathbf{A}$ is symmetric and positive definite. 
Let $\mathbf{A} = \mathbf{C}\mathbf{C}^T$ with $\mathbf{C}$ being a lower triangular matrix
with positive diagonal entries $\mathbf{C}_{jj} > 0$.
There exists a lower triangular blockwise rank-$r$ matrix $\mathbf{C_{\H}}$ such that 
\bean
\norm{\mathbf{A}-\mathbf{C_{\H}}\mathbf{C_{\H}}^T}_2 &\leq& \Big(C_{\rm Ch} h^{-1} {\rm depth}(\mathbb{T}_{\mathcal{I}}) e^{-br^{1/(d+1)}} \\
 & & \;\; + C_{\rm Ch}^2 h^{-2} {\rm depth}(\mathbb{T}_{\mathcal{I}})^2 e^{-2br^{1/(d+1)}}\Big)\norm{\mathbf{A}}_2, \nonumber
\eean
where $C_{\rm Ch} = 2C_{\rm sp}C_{\rm apx}\sqrt{\kappa_2(\mathbf{A})}$, with the constant $C_{\rm apx}$ 
from Theorem~\ref{th:blockapprox} and the spectral condition number $\kappa_2(\mathbf{A})$. 
\end{corollary}
\bpro
Since $\mathbf{A}$ is symmetric and positive definite, the Schur complements $\mathbf{S}(\tau,\tau)$ 
are symmetric and positive definite as well and 
therefore we get $\mathbf{U}(\tau) = \mathbf{C}(\tau)^T$ in \eqref{eq:LUdefinition}.
Moreover, we have $\norm{\mathbf{A}}_2=\norm{\mathbf{C}}_2^2$
and $\kappa_2(\mathbf{C}) = \norm{\mathbf{C}^{-1}}_2\norm{\mathbf{C}}_2 = \sqrt{\kappa_2(\mathbf{A})}$.
\epro

\section{Numerical Examples}
In this section, we present some numerical examples in two and three dimensions to confirm
our theoretical estimates derived in the previous sections.
Since numerical examples for the Dirichlet case have been studied before, e.g. in \cite{GrasedyckDissertation,BebendorfHackbusch}, 
we will focus on mixed Dirichlet-Neumann and pure Neumann problems in two and three dimensions. 

With the choice $\eta=2$ for the admissibility parameter in \eqref{eq:admissibility},
the clustering is done by the standard geometric clustering algorithm,
i.e., by splitting bounding boxes in half until they are admissible or smaller than the constant $n_{\text{leaf}}$,
which we choose as $n_{\text{leaf}} = 25$ for our computations.
An approximation to the inverse Galerkin matrix is computed by using the bestapproximation via singular value
decomposition. 
Throughout, we use the C-library HLiB \cite{HLib} developed at the 
Max-Planck-Institute for Mathematics in the Sciences.\\ 

\subsection{2D-Diffusion}
As a model geometry, we consider the unit square 
$\Omega = (0,1)^2$.
The boundary $\Gamma = \partial \Omega$ is divided into the
 Neumann part $\Gamma_{D} := \{\mathbf{x} \in \Gamma \,:\, \mathbf{x}_1 = 0 \vee \mathbf{x}_2 = 0 \}$ and the  
Dirichlet part $\Gamma_{\mathcal{N}} = \Gamma \backslash \overline{\Gamma_{D}}$. 
We consider the bilinear form $a(\cdot,\cdot):H^1_0(\Omega,\Gamma_D)\times H^1_0(\Omega,\Gamma_D) \ra \R$
corresponding to the mixed Dirichlet-Neumann Poisson problem
\bean\label{eq:numericsMixed}
a(u,v) := \skp{\nabla u,\nabla v}_{L^2(\Omega)}
\eean
and use a lowest order Galerkin discretization in $S^{1,1}_0(\T_h,\Gamma_D)$. 

As a second example, we study pure Neumann boundary conditions, i.e. $\Gamma =\Gamma_{\mathcal{N}}$,
and use the bilinear form $a_{\mathcal{N}}(\cdot,\cdot): H^1(\Omega) \times H^1(\Omega)\ra \R$
corresponding to the stabilized Neumann Poisson problem 
\been\label{eq:numericsNeumann}
a_{\mathcal{N}}(u,v) := \skp{\nabla u, \nabla v} + \skp{u,1}\skp{v,1}
\een
and a lowest order Galerkin discretization in $S^{1,1}(\T_h)$. \\

In Figure \ref{fig:2DDiff}, we compare the decrease of the upper bound 
$\norm{\mathbf{I}-\mathbf{A}\mathbf{B}_{\H}}_2$ of the relative error with the increase 
in the block-rank for a fixed number $N = 262.144$ of degrees of freedom, where 
the largest block of $\mathbf{B}_{\H}$ has a size of 32.768.

\begin{figure}[h]
\begin{center}
\psfrag{Error}[c][c]{%
 \footnotesize  Error}
\psfrag{Block rank r }[c][c]{%
 \footnotesize  Block rank r}
\psfrag{asd}[l][l]{\footnotesize $\exp(-1.2\, r)$}
\psfrag{jkl}[l][l]{\footnotesize $\norm{I-AB_{\H}}_2$}
\includegraphics[width=0.48\textwidth]{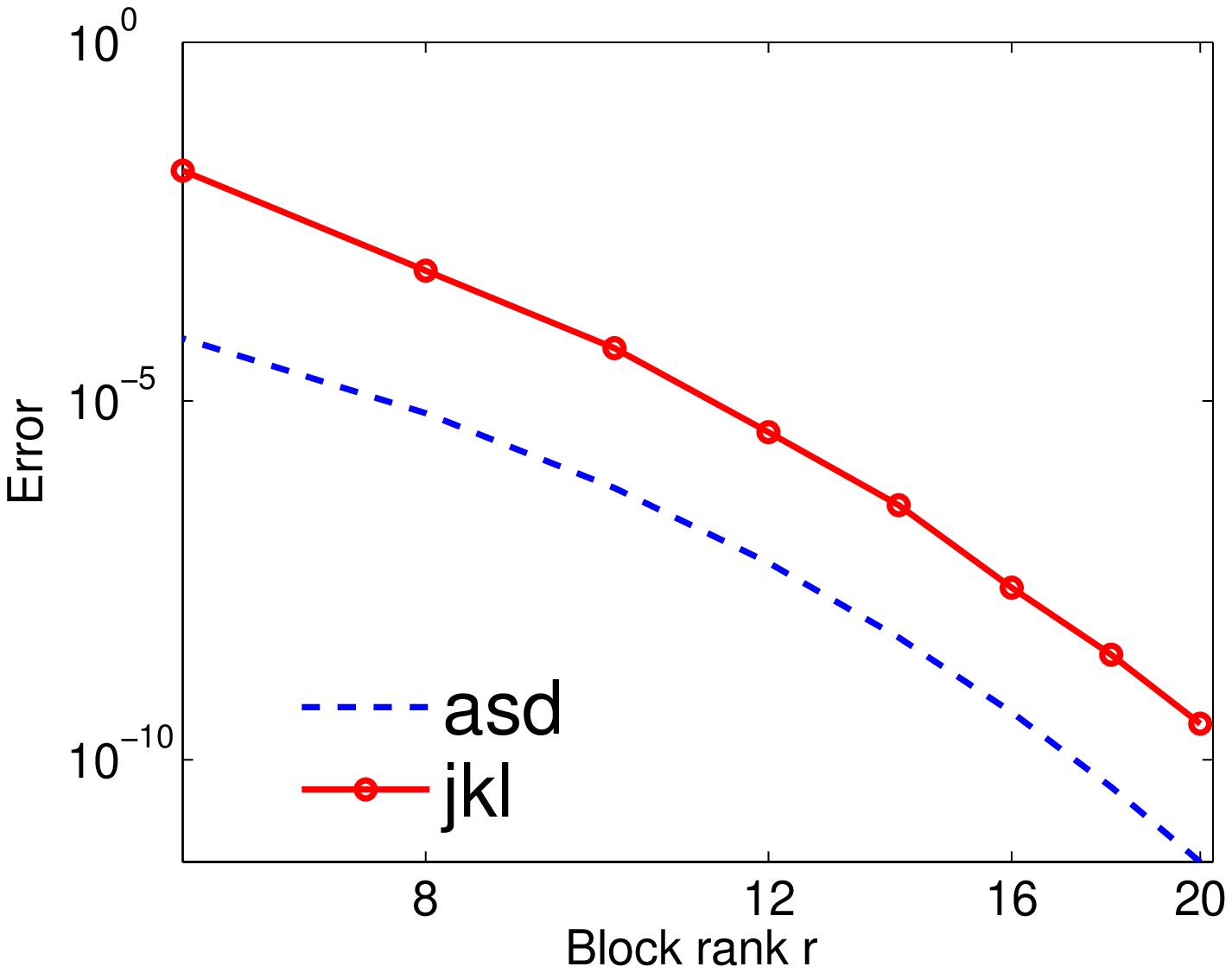}
\psfrag{asd}[l][l]{\footnotesize $\exp(-1.3\, r)$}
\includegraphics[width=0.48\textwidth]{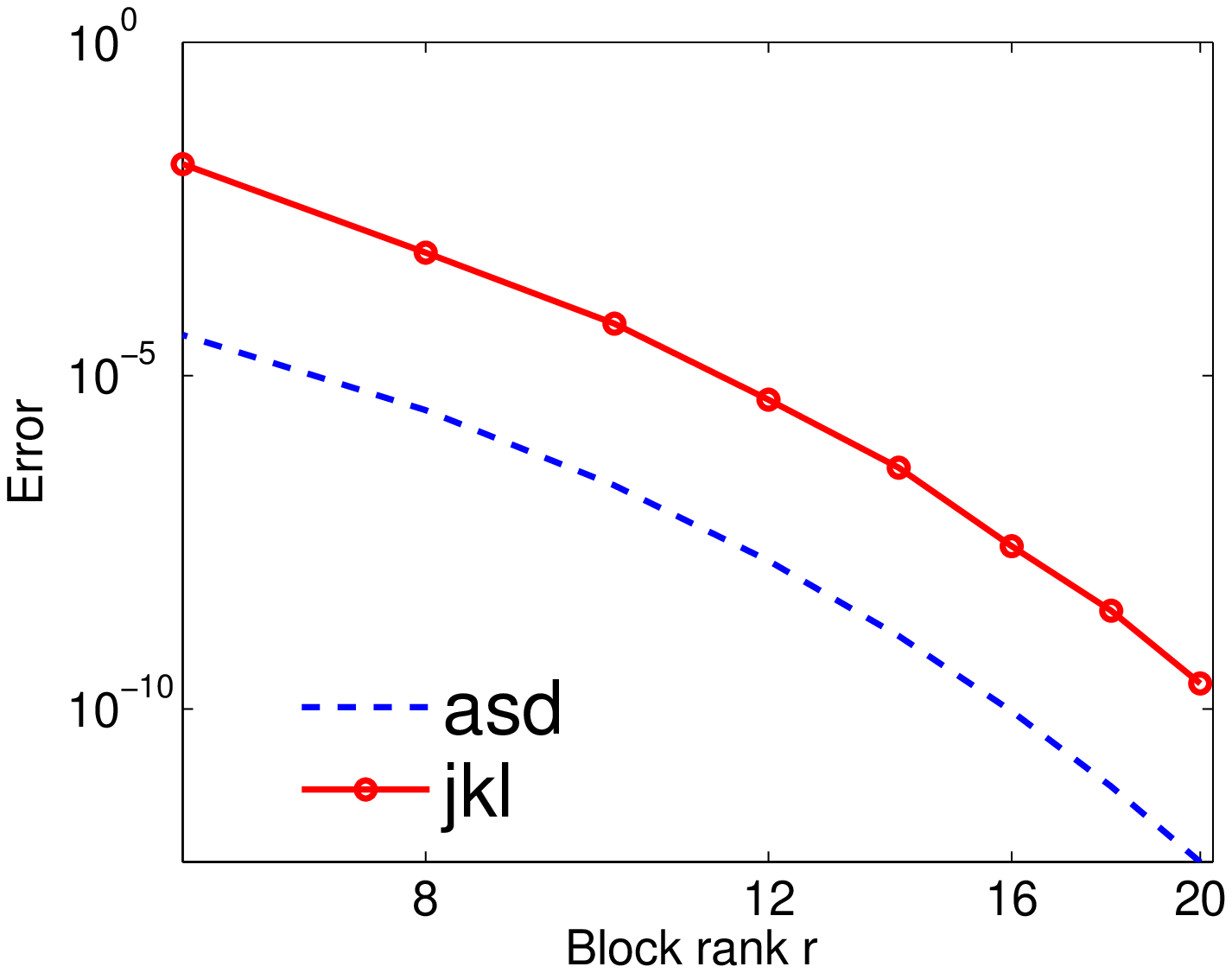}
\caption{\footnotesize Mixed boundary value problem (left), pure Neumann boundary value problem (right) in 2D.}
\label{fig:2DDiff}
\end{center}
\end{figure}

As one can see, we observe exponential convergence in the block rank, where the convergence rate 
is $\exp(-br)$, which is even faster than the rate of $\exp(-br^{1/3})$ guaranteed by Theorem~\ref{th:Happrox}.

\subsection{3D-Diffusion}
For our three dimensional example, we consider the unit cube $\Omega = (0,1)^3$ with the Dirichlet boundary 
$\Gamma_{D} := \{\mathbf{x} \in \Gamma \,:\, \exists i\in\{1,2,3\} : \mathbf{x}_i = 0\}$ and the Neumann part
$\Gamma_{\mathcal{N}} = \Gamma \backslash \overline{\Gamma_{D}}$. 

Again, we consider the bilinear forms \eqref{eq:numericsMixed} and \eqref{eq:numericsNeumann}
corresponding to the weak formulations of the Dirichlet-Neumann Poisson problem and the stabilized
Neumann problem. \\

In Figure \ref{fig:3DDiff}, we  compare the decrease of 
$\norm{\mathbf{I}-\mathbf{A}\mathbf{B}_{\H}}_2$ with the increase 
in the block-rank for a fixed number $N = 32.768$ of degrees of freedom, where 
the largest block of $\mathbf{B}_{\H}$ has a size of 4.096. 

\begin{figure}[h]
\begin{center}
\psfrag{Error}[c][c]{%
 \footnotesize  Error}
\psfrag{Block rank r }[c][c]{%
 \footnotesize  Block rank r}
\psfrag{asd}[l][l]{\footnotesize $\exp(-2.3\, r^{1/2})$}
\psfrag{jkl}[l][l]{\footnotesize $\norm{I-AB_{\H}}_2$}
\includegraphics[width=0.48\textwidth]{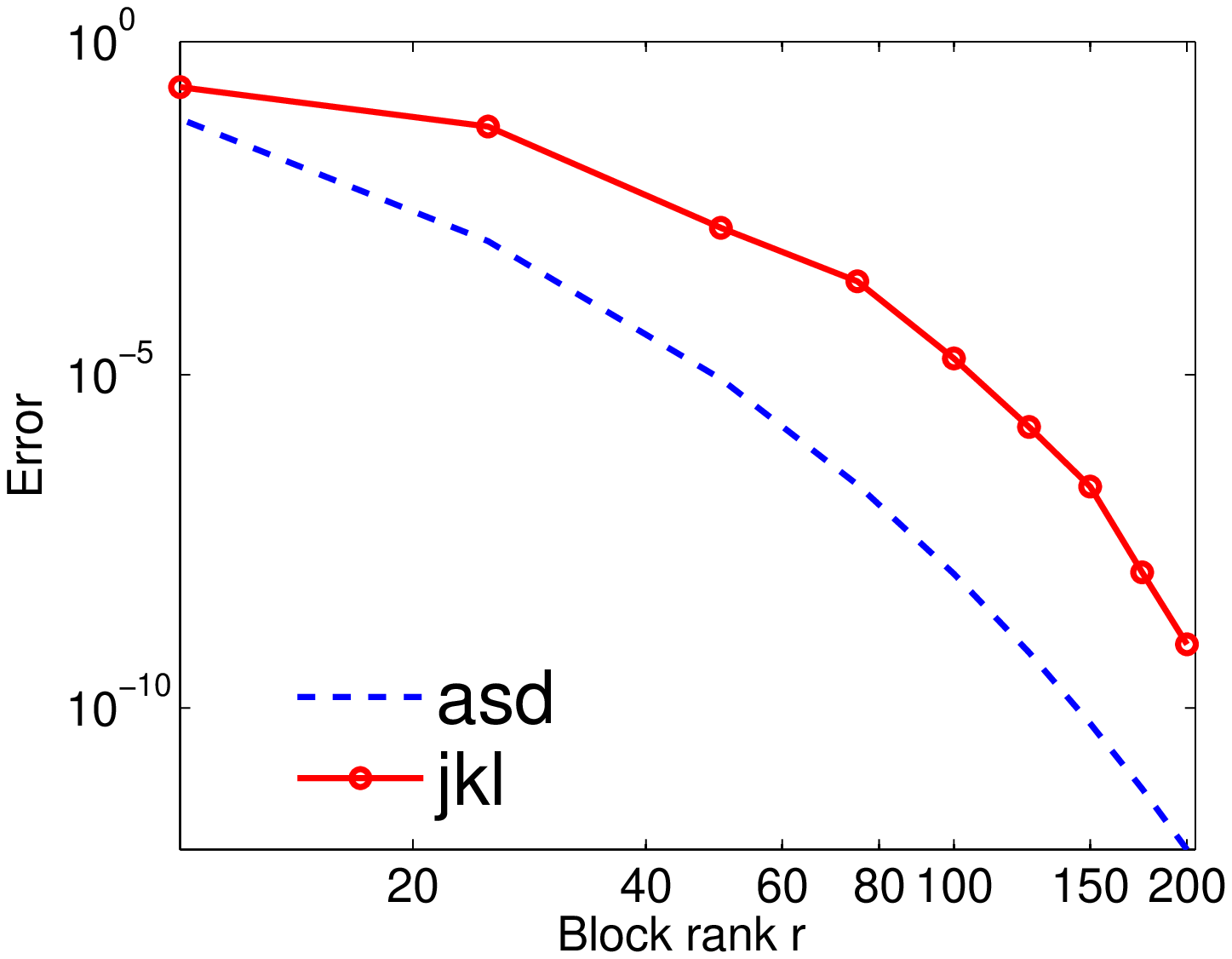}
\psfrag{asd}[l][l]{\footnotesize $\exp(-2.5\, r^{1/2})$}
\includegraphics[width=0.48\textwidth]{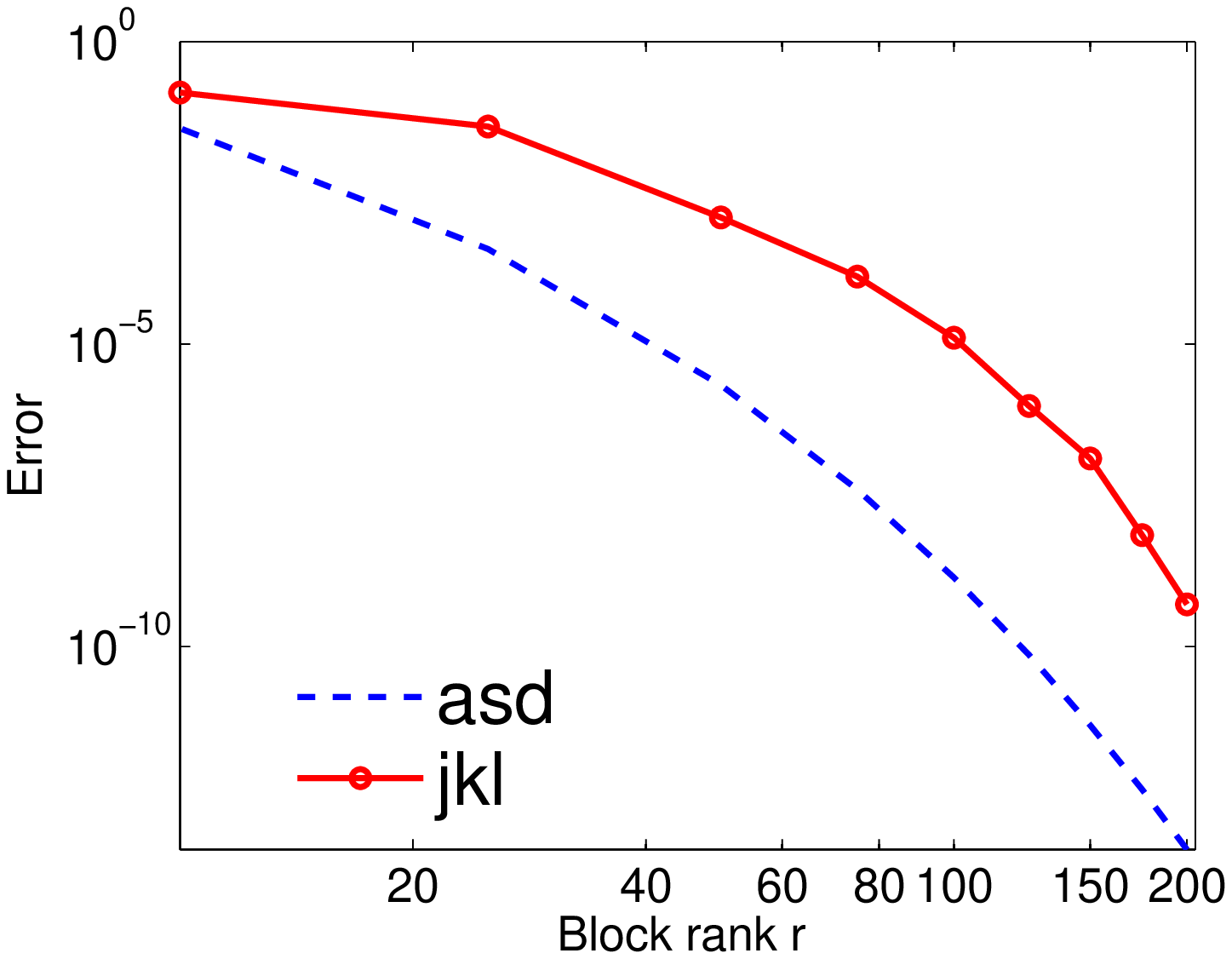}
\caption{\footnotesize Mixed boundary value problem (left), pure Neumann boundary value problem (right) in 3D}
\label{fig:3DDiff}
\end{center}
\end{figure}

Comparing the results with our theoretical bound from Theorem~\ref{th:Happrox},
 we empirically observe a rate of $e^{-br^{1/2}}$ instead
of $e^{-br^{1/4}}$.
Moreover, whether we study mixed boundary conditions or pure Neumann boundary conditions does not make any difference,
as both model problems lead to similar computational results.

\subsection{Convection-Diffusion}
Finally, we study a convection-diffusion problem on the L-shaped domain 
$\Omega = (0,1)\times (0,\frac{1}{2}) \cup (0,\frac{1}{2})\times[\frac{1}{2},1)$.
The boundary $\Gamma = \partial \Omega$ is divided into the
 Neumann part $\Gamma_{\mathcal{N}} := \{\mathbf{x} \in \Gamma \,:\, \mathbf{x}_2 = 0 \; \vee \; \mathbf{x}_1 = 1 \}$ and the  
Dirichlet part $\Gamma_D = \Gamma \backslash \overline{\Gamma_{\mathcal{N}}}$. 

We consider the bilinear form $a(\cdot,\cdot):H^1_0(\Omega,\Gamma_D)\times H^1_0(\Omega,\Gamma_D) \ra \R$
corresponding to the mixed Dirichlet-Neumann Poisson problem
\bea
a(u,v) := c\skp{\nabla u,\nabla v}_{L^2(\Omega)} + \skp{\mathbf{b}\cdot \nabla u, v}_{L^2(\Omega)}
\eea
with $c = 10^{-2}$ and $\mathbf{b}(x_1,x_2) = (-x_2,x_1)^T$
and use a lowest order Galerkin discretization in $S^{1,1}_0(\T_h,\Gamma_D)$. 

In Figure \ref{fig:ConvDiff}, we observe exponential convergence of the upper bound $\norm{\mathbf{I}-\mathbf{A}\mathbf{B}_{\H}}_2$
of the relative error with respect to the increase 
in the block-rank for a fixed number $N = 196.352$ of degrees of freedom, where 
the largest block of $\mathbf{B}_{\H}$ has a size of 24.544. 

\begin{figure}[h]
\begin{center}
\psfrag{Error}[c][c]{%
 \footnotesize  Error}
\psfrag{Block rank r }[c][c]{%
 \footnotesize  Block rank r}
\psfrag{asd}[l][l]{\footnotesize $\exp(-1.08\, r)$}
\psfrag{jkl}[l][l]{\footnotesize $\norm{I-AB_{\H}}_2$}
\includegraphics[width=0.48\textwidth]{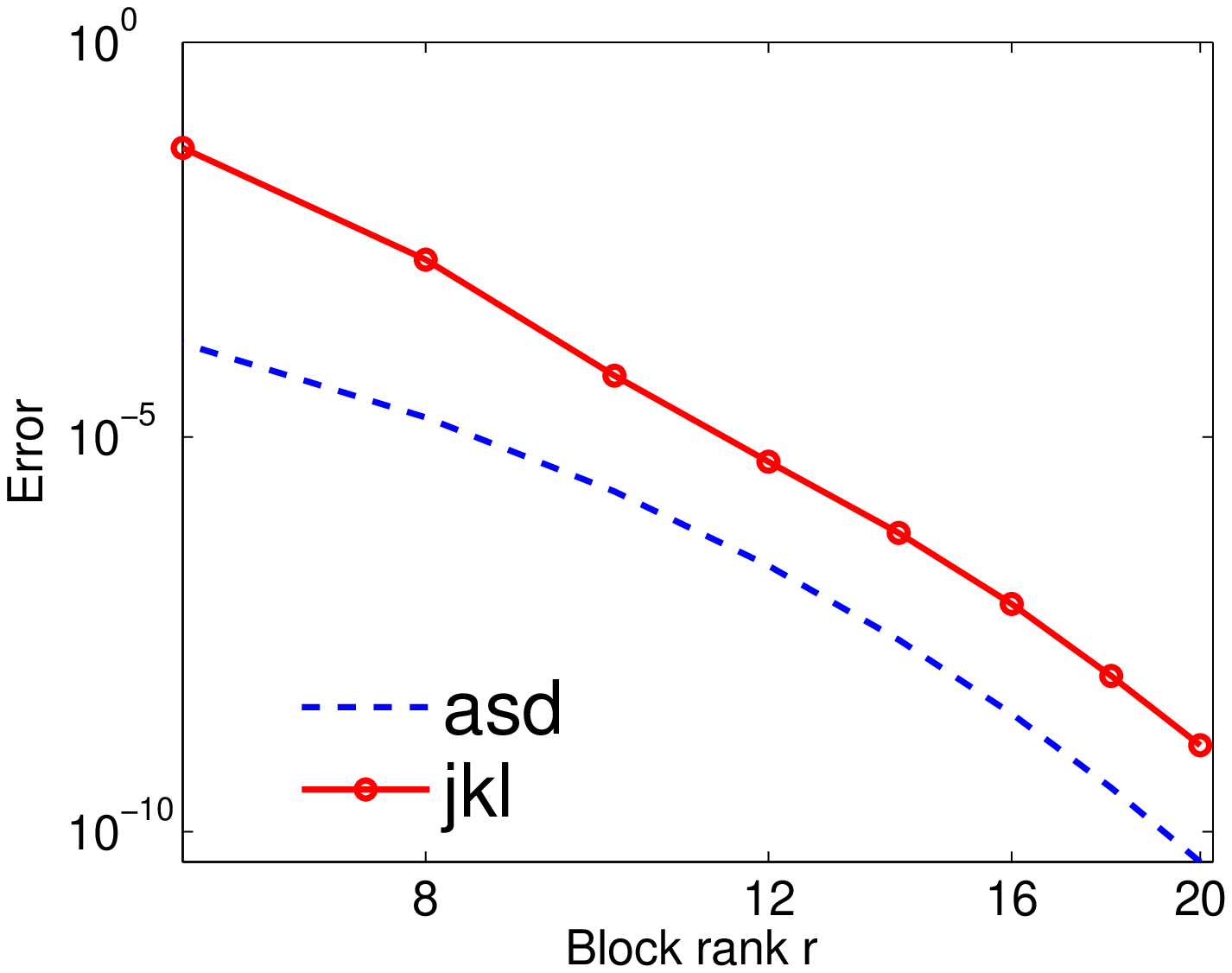}
\psfrag{asd}[l][l]{\footnotesize $\exp(-1.1\, r)$}
\includegraphics[width=0.48\textwidth]{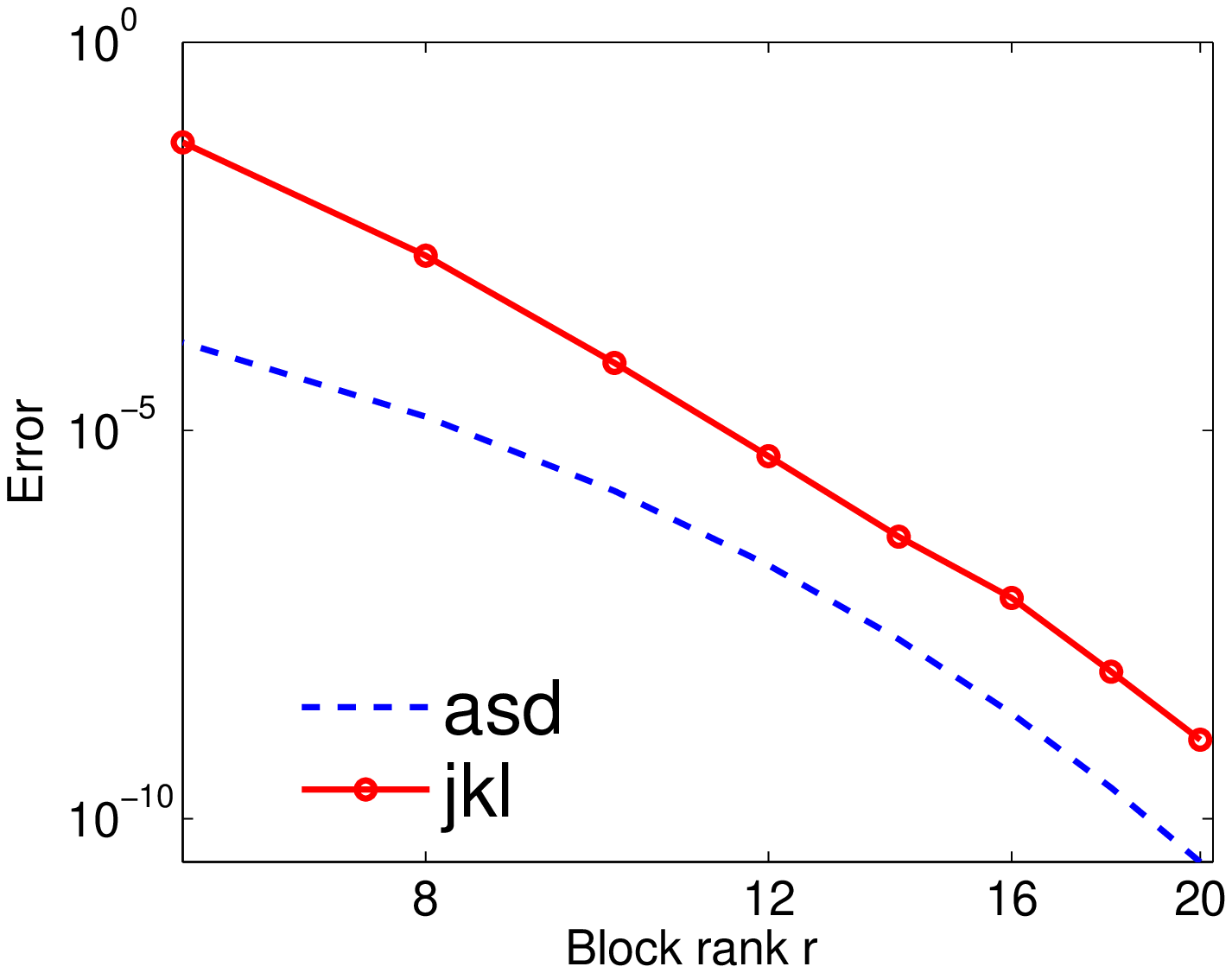}
\caption{\footnotesize 2D Convection-Diffusion: Mixed boundary value problem (left), pure Neumann boundary value problem (right).}
\label{fig:ConvDiff}
\end{center}
\end{figure}

\nocite{*}
\bibliography{bibliography_2}{}
\bibliographystyle{amsalpha}

\end{document}